\documentclass[11pt]{article}

\usepackage{amsfonts,amsmath,amssymb,amsthm, color, enumerate,graphicx}
\usepackage{hyperref}

\usepackage{tikz}
\usetikzlibrary{calc}
\usepackage{fullpage}
\usepackage[numbers]{natbib}
\newtheorem{theorem}{Theorem}

\newtheorem{lemma}[theorem]{Lemma}
\newtheorem{corollary}[theorem]{Corollary}

\newtheorem{remark}[theorem]{Remark}
\newtheorem{claim}[theorem]{Claim}

\newtheorem{proposition}[theorem]{Proposition}

\numberwithin{theorem}{section}

\begin{document}
 \date{}

\title{Codegree thresholds for covering $3$-uniform hypergraphs}

\author{Victor Falgas-Ravry \\
\small Department of Mathematics\\ [-0.8ex]
\small Vanderbilt University, Nashville\\ [-0.8ex]
\small Tennessee, U.S.A.\\
\small \tt victor.falgas-ravry@vanderbilt.edu\\
\and
Yi Zhao \thanks{Research partially supported by NSF Grant DMS-1400073.}\\
\small Department of Mathematics\\[-0.8ex]
\small Georgia State University, Atlanta\\[-0.8ex]
\small Georgia, U.S.A.\\
\small \tt yzhao6@gsu.edu
}

\maketitle
%\vspace{-0.5in}
\begin{abstract}  
Given two $3$-uniform hypergraphs $F$ and $G$, we say that $G$ has an \emph{$F$-covering} if we can cover $V(G)$ by copies of $F$. The \emph{minimum codegree} of $G$ is the largest integer $d$ such that every pair of vertices from $V(G)$ is contained in at least $d$ triples from $E(G)$. Define $c_2(n,F)$ to be the largest minimum codegree among all $n$-vertex 3-graphs $G$ that contain no $F$-covering.
This is a natural problem intermediate (but distinct) from the well-studied Tur\'an problems and tiling problems.
%determining minimum codegree thresholds for the existence of a single copy of $F$ on one hand, and for the existence of an $F$-tiling on the other hand.
In this paper, we determine $c_2(n, K_4)$ (for $n>98$) and the associated extremal configurations (for $n>998$), where $K_4$ denotes the complete $3$-graph on $4$ vertices. We also obtain bounds on $c_2(n,F)$ which are apart by at most $2$ in the cases where $F$ is $K_4^ -$ ($K_4$ with one edge removed), $K_5^-$, and the tight cycle $C_5$ on $5$ vertices. 
\end{abstract}

\section{Introduction}
\subsection{Notation}
Given a set $A$ and a positive integer $k$, we write $A^{(k)}$ for the collection of $k$-element subsets of $A$. We use $[n]$ as a shorthand for the collection of the first $n$ natural numbers, $[n]=\{1,2, \ldots, n\}$. We shall often consider pairs or triples of vertices; when there is no risk of confusion, we write $ab$ and $abc$ as a shorthand for $\{a,b\}$ and $\{a,b,c\}$ respectively.
A \emph{$k$-uniform hypergraph}, or \emph{$k$-graph}, is a pair $G=(V, E)$, where $V$ is a set of \emph{vertices}, and $E \subseteq V^{(k)}$ is a collection of $k$-subsets of $V$, which form the \emph{edges} of $G$. 
A \emph{subgraph}  of $G$ is a $k$-graph $H$ with $V(H) \subseteq V(G)$ and $E(H) \subseteq E(G)$. The \emph{degree} of a vertex $x\in V(G)$, which we denote by $d(x)$, is the number of edges of $G$ containing $x$. The \emph{minimum degree} $\delta_1(G)$ of $G$ is the minimum of $d(x)$ over all vertices $x\in V(G)$. 

In this paper, we will focus on 3-graphs $G=(V,E)$ and another degree-like quantity, and its minimum: the \emph{codegree} of a pair $xy\in V^{(2)}$, denoted by $d(x,y)$, is the number of edges of $G$ containing the pair $xy$. We write $\Gamma(y,z)$ for the \emph{neighbourhood} of the pair $xy$, i.e. the set of $z\in V\setminus\{x,y\}$ such that $xyz\in E(G)$. The \emph{minimum codegree} of $G$ is $\delta_2(G)=\min_{xy\in V^{(2)}}d(x,y)$. The \emph{link graph} of a vertex $x \in V(G)$ is the collection $G_x$ of all pairs $uv$ such that $xuv\in E(G)$. The \emph{degree} of $u$ in $G_x$ is the number of vertices $v$ such that $uv\in G_x$; note this is exactly the codegree of $x$ and $u$. Finally we define the \emph{edit distance} between two $3$-graphs $G$ and $G'$ on the same vertex set to be the minimum number of changes required to make $G$ isomorphic to $G'$, where a change consists in replacing an edge by a non-edge and vice-versa. 

\subsection{The problem}
Let $F$ be a fixed $3$-graph on $t$ vertices with at least one $3$-edge. A $3$-graph $G$ has an \emph{$F$-covering} if we can cover $V(G)$ with \emph{$F$-subgraphs} (subgraphs that are isomorphic to $F$). For $n \geq t$ and $i=1, 2$, we define 
\[c_i(n, F) = \max\{\delta_i(G): \  \vert V(G) \vert = n \textrm{ and }G \textrm{ does not have an $F$-covering}\}.\]
and call $c_1(n, F)$ the \emph{covering degree-threshold} and $c_2(n, F)$ the \emph{covering codegree-threshold} of $F$.

The covering threshold $c_i(n, F)$ was introduced by Han, Zang and Zhao~\cite{HanZangZhao15} when they studied the minimum degree that guarantees the existence of a $K$-tiling, where $K$ is a complete $3$-partite $3$-graph. It was shown implicitly in~\cite{HanZangZhao15} that $c_1(n,K)=(6-4\sqrt{2}+o(1))\binom{n}{2}$ if $K$ has at least two vertices in each part (in contrast, it is easy to see that $c_1(n,K)= o(n^2)$ if some part of $K$ has only one vertex).
It was also noted that $c_1(n, F)= ( 1 - 1/ (\chi(F) - 1) + o(1) )n$ for all \emph{graphs} $F$, where $\chi(F)$ is the chromatic number of $F$.

Our objective in this paper is to study the behaviour of the function $c_2(n, F)$ for various $3$-graphs $F$. In other words, we seek to determine what codegree condition is necessary to guarantee that \emph{all} vertices in a $3$-graph $G$ are contained in copies of $F$. When determining the exact value of $c_2(F,n)$ is difficult, we may ask instead for its asymptotic behaviour. It can be shown (see Section~\ref{section: preliminaries}) that the limit 
%\[c_1(F) = \lim_{n\rightarrow \infty}c_1(n,F)\Big/ \binom{n-1}{2} \qquad \textrm{and}\qquad  c_2(F) = \lim_{n\rightarrow \infty}c_2(n,F)\Big/ (n-2)\]
\[c_2(F) = \lim_{n\rightarrow \infty}\frac{c_2(n,F)}{n-2}\]
exists.\footnote{This is a direct corollary of the proof of Proposition~6 from~\cite{FalgasRavryMarchantPikhurkoVaughan15} on the existence of conditional codegree density (with an uncovered vertex $x$ used as the conditional subgraph $H$), or can be proved in the same way as the existence of the usual codegree density $\gamma(F)$ in~\cite{MubayiZhao07}.}  We call $c_2(F)$ the \emph{covering codegree-density} of $F$.

Let us introduce the $3$-graphs relevant to the present work. 
Let $K_t=([t], [t]^{(3)})$ denote the complete $3$-graph on $t$ vertices, and let $K_t^-$ denote the $3$-graph obtained from $K_t$ by removing one $3$-edge. The \emph{strong} or \emph{tight} $t$-cycle is the $3$-graph $C_t$ on $[t]$ with $3$-edges $\{123, 234, 345, \ldots, (t-2)(t-1)t, (t-1)t1, t12\}$.  We denote by $F_{3,2}$ the $3$-graph $([5],\{123,124,125,345\})$. Finally a \emph{Steiner Triple System} (STS) is a $3$-graph in which every pair of vertices is contained in exactly one $3$-edge; it is a 168 years old result of Kirkman~\cite{Kirkman1847} that a STS on $t$ vertices exists if and only if $t\equiv 1,3 \mod 6$. The \emph{Fano plane} is the unique (up to isomorphism) STS on $7$ vertices, which we denote by $\mathrm{Fano}$. 

%\begin{problem}\label{problem: covering degree-density}
%Determine $c_1(F)$.
%\end{problem}
%\begin{problem}\label{problem: covering codegree-density}
%Determine $c_2(F)$.
%\end{problem}
\subsection{Motivation and related work in extremal hypergraph theory}\label{subsection: literature}

Before we state our results, let us give some motivation and background for our problem.  Let $F$ be a fixed $3$-graph on $t$ vertices with at least one $3$-edge. A $3$-graph $G$ is \emph{$F$-free} if it does not contain a copy of $F$ as a subgraph. Further $G$ has an \emph{$F$-tiling}, or \emph{$F$-factor}, if we can cover $V(G)$ with \emph{vertex-disjoint} $F$-subgraphs. There has been much research into the degree and/or codegree conditions needed ensure the existence of an $F$-subgraph or of an $F$-factor in a $3$-graph $G$. Determining the degree/codegree condition necessary to guarantee an $F$-covering is intermediate between these two well-studied problems. As we show in the next subsection, the existence, covering, and tiling problems give rise to different thresholds in their codegree versions, so that our work is novel. It is hoped that studying the properties of the covering codegree threshold function $c_2(n,F)$ --- such as \emph{supersaturation}, discussed in Section~\ref{section: conclusion}, which could be useful for applying semi-random methods to tiling problems --- will lead to insights about both the existence and tiling problems.

%\textbf{-Behaviour for $3$-graphs different from behaviour for graphs \\
%-Tur\'an density, codegree density , existence of an $F$-covering, and existence of an $F$-tiling all give rise to different thresholds; in particular not covering a given vertex is in general \emph{not} the obstruction to finding an $F$-tiling, nor does the appearance of an $F$-subgraph coincide with that of an $F$-covering \\
%-Motivation from Yi's other paper\\
%-Need to add some examples about tiling thresholds at the end of the preliminaries section (currently we only show the extremal and covering problems are distinct)\\}

The \emph{Tur\'an number} $\mathrm{ex}(n, F)$ of $F$ is the maximum number of $3$-edges an $F$-free $3$-graph on $n$ vertices can have. %The \emph{degree threshold} $\mathrm{degex}(n,F)$ of $F$ is the maximum of $\delta_1(G)$ over all $F$-free $3$-graphs on $n$ vertices, while 
The \emph{codegree threshold} $\mathrm{ex}_2(n, F)$ of $F$ is the maximum of $\delta_2(G)$ over all $F$-free $3$-graphs $G$ on $n$ vertices. It is well-known that $\mathrm{ex}(n,F)/\binom{n}{3}$ tends to a limit $\pi(F)$ as $n\rightarrow \infty$; this limit is known as the \emph{Tur\'an density} of $F$. Similarly, $\mathrm{ex}_2(n, F)/(n-2)$ tends to a limit $\gamma(F)$ called the \emph{codegree density} or \emph{$2$-Tur\'an density} of $F$ as $n\rightarrow \infty$. The extremal theory of $3$-graphs and within it the study of Tur\'an-type problems have received extensive attention from the combinatorics community since the 1950s , with strenuous efforts devoted in particular to the (in)famous and still-open conjecture of Tur\'an~\cite{Turan41} that $\pi(K_4)=5/9$. 
%Despite this, almost no non-trivial Tur\'an densities for $3$-graphs were known until de Caen and F\"uredi~\cite{deCaenFuredi00} determined $\pi(\mathrm{Fano})$ in 2000. A few more Tur\'an densities have been determined since, 
See the surveys of F\"uredi~\cite{Furedi91} and Keevash~\cite{Keevash11} for an overview of results. 
There has been significant interest in other extremal quantities, and in particular in codegree densities for $3$-graphs. The first result on codegree density was due to Mubayi~\cite{Mubayi05}, who showed $\gamma(\mathrm{Fano})=\frac{1}{2}$. Keevash and Zhao~\cite{KeevashZhao07} determined the codegree densities of some projective geometries, which included the Fano plane as a special case. The codegree threshold for the Fano plane was determined by Keevash~\cite{Keevash09} via hypergraph regularity and later by DeBiasio and Jiang~\cite{DeBiasioJiang12} by direct combinatorial means.  Mubayi and Zhao~\cite{MubayiZhao07} studied general properties of codegree density, while Falgas-Ravry~\cite{FalgasRavry13} gave examples of non-isomorphic lower bound constructions for $\gamma(K_t)$. More recently Falgas--Ravry, Marchant, Pikhurko and Vaughan~\cite{FalgasRavryMarchantPikhurkoVaughan15} determined the codegree threshold of $F_{3,2}$, and Falgas-Ravry, Pikhurko and Vaughan~\cite{FalgasRavryPikhurkoVaughan15+} showed $\gamma(K_4^-)=\frac{1}{4}$ via a flag algebra computation, resolving a conjecture of Czygrinow and Nagle~\cite{Nagle99}. Another conjecture of Czygrinow and Nagle\cite{CzygrinowNagle01} remains open, namely that $\gamma(K_4)=\frac{1}{2}$. Certainly $\gamma(F)\leq c_2(F)$ for any $3$-graph $F$, and it may be hoped that giving good upper bounds for the latter may also help bounding the former.

%de Caen and F\"uredi~\cite{deCaenFuredi00} showed in 2000 that $\pi(\mathrm{Fano plane})=3/4$ (a notable exception being a 1974 result of Bollob\'as~\cite{Bollobas74}).
% The Tur\'an number of the Fano plane was determined shortly afterwards by F\"uredi--Simonovits~\cite{FurediSimonovits05}, and Keevash--Sudakov~\cite{KeevashSudakov05}. F\"uredi, Pikhurko and Simonovits subsequently determined the Tur\'an number of $F_{3,2}$~\cite{FurediPikhurkoSimonovits05}. An important further development was the advent of Razborov's flag algebra method~\cite{Razborov07}, which led to significant improvements of known upper bounds in a variety of cases. In particular it is now known that
%\begin{align*}
%&\frac{2}{7}\leq \pi(K_4^-) \leq 0.286889, \qquad 2\sqrt{3}-3\leq \pi(C_5) \leq 0.468287\\
%&\frac{5}{9}\leq \pi(K_4) \leq 0.561666, \quad \textrm{ and } \quad  \frac{3}{4}\leq \pi(K_5) \leq 0.769533.
%\end{align*} 
%The lower bounds on $\pi(K_4^-)$ and $\pi(C_5)$ come from iterated constructions due to Frankl--F\"uredi~\cite{FranklFuredi84} and Mubayi--R\"odl~\cite{MubayiRodl02} respectively, while the lower bounds on $\pi(K_4)$ and $\pi(K_5)$ can be attained in many different ways, see the survey of Sidorenko~\cite{Sidorenko95}; the upper bounds are from flag algebra computations in Falgas-Ravry--Vaughan ~\cite{FalgasRavryVaughan13}.

In addition to these Tur\'an-type problems, there has been much research activity on the problem of determining thresholds for the existence of $F$-factors. The situation for ordinary ($2$-)graphs is now well-understood: the celebrated Hajnal--Szemer\'edi theorem~\cite{HajnalSzemeredi70} gives the exact minimum degree condition guaranteeing the existence of $F$-factors in an $n$-vertex graph when $F$ is a clique, while K\"uhn and Osthus~\cite{KuhnOsthus09} determined the minimum degree condition for general graphs $F$  up to an additive constant. %This stands in sharp contrast with the state of affairs for $k$-graphs when $k\ge 3$.
On the other hand, until recently not much was known about tiling for $k$-graphs when $k\ge 3$. While there has been a spate of results in the last few years, see~\cite{czygrinow2015, CzygrinowDeBiasioNagle14, GaoHan15, HanLoTreglownZhao15, KeevashMycroft14, KuhnOsthus06, LoMarkstrom13, LoMarkstrom11, Mycroft14, RodlRucinskiSzemeredi09}, many more open problems remain. We refer to the surveys of R\"odl and Ruci\'nski~\cite{RodlRucinski11} and Zhao~\cite{Zhao15} for a more detailed discussion of the area, and briefly mention below four results relevant to the present work. For $i\in\{1,2\}$ and $n\equiv 0 \mod \vert V(F)\vert$, let 
\[t_i(n, F) = \max\{\delta_i(G): \  \vert V(G) \vert = n \textrm{ and }G \textrm{ does not have an $F$-factor}\}.\]
%let  $t_{i}(n,F)$  be the minimum codegre required to guarantee the existence of an $F$-factor in an $n$-vertex $3$-graph, provided $n\equiv 0 \mod \vert V(F)\vert$. 
Trivially $c_i(n, F)\leq t_i(n,F)$ for any $3$-graph $F$ with at least one edge. 
%Let also $t_{k-1}(n,F)$ denote the appropriately defined $k$-graph analogue of the tiling codegree threshold.
%and any $n$ divisible by $\vert V(F)\vert $.
%One of the key tools to emerge in the study of these problems is the \emph{absorbing method} of R\"odl, Ruci\'nski and Szemer\'edi~\cite{RodlRucinskiSzemeredi06}. 
%R\"odl, Ruci\'nski and Szemer\'edi~\cite{RodlRucinskiSzemeredi09} determined the codegree threshold $t_{k-1}(n, k$-edge$)$ for the existence of a perfect matching in a $k$-graph. 
%A few (though not many) more results for tiling codegree thresholds are known, as well as plenty of open problems; we 
Lo and Markstr\"om~\cite{LoMarkstrom11, LoMarkstrom13} determined $t_2(n,F)$ asymptotically when $F=K_4$ and $F=K_4^-$. Independently Keevash and Mycroft~\cite{KeevashMycroft14} determined $t_2(n, K_4)$ exactly, and recently Han, Lo, Treglown and Zhao~\cite{HanLoTreglownZhao15} determined $t_2(n, K_4^-)$ exactly as well, in both cases for $n$ sufficiently large. 
%Czygrinow, DeBiasio and Nagle~\cite{CzygrinowDeBiasioNagle14} computed $t_2(n,F)$ exactly (again for $n$ large enough) when $F$ is the ($3$-uniform) loose cycle $S_2$ of length $2$. 
%Gao and Han~\cite{GaoHan15} determined $t_2(n, S_3)$ exactly for large $n$, where $S_3$ is the loose cycle of length $3$. 
%Independently, Czygrinow~\cite{czygrinow2015} determined $t_2(n, S_t)$ for all $t\ge 3$, where $S_t$ is the loose cycle of length $t$. Finally, Mycroft~\cite{Mycroft14} recently determined $t_{k-1}(n,F)$ asymptotically for a large class of $k$-partite $k$-graphs $F$, including all complete $k$-partite $k$-graphs and all loose cycles of fixed length, giving asymptotic versions of some of the aforementioned results. 
Finally in~\cite{HanZangZhao15} Han, Zang and Zhao asymptotically determined $t_1(n, K)$ for all complete $3$-partite $3$-graphs $K$. In particular, they showed that $t_1(n, K)= c_1(n,K)=(6-4\sqrt{2}+o(1))\binom{n}{2}$ for \emph{certain} $K$. This gives further motivation for the present paper: by determining $c_2(n,F)$ for $3$-graphs $F$, we may hope likewise to shed light on $t_2(n, F)$ and facilitate its (asymptotic) computation.

\subsection{Results}\label{subsection: results}
In this paper, we determine the codegree covering threshold for $K_4$ for sufficiently large $n$. 

%What is more, we obtain stability and fully characterize the extremal configurations: every configuration which is near-extremal are close to a certain $3$-graph $F_1(n)$ (defined in Section~\ref{subsection: k4 covering codegree density}) in the edit distance, and the extremal configurations are covered (up to isomorphism) by a family of $\theta(n)$ non-isomorphic saturated $3$-graphs obtained by perturbing $F_1(n)$ slightly.

\begin{theorem} \label{theorem: K4 threshold} % YZ removed [Covering codegree threshold for $K_4$]
For every $n \in \mathbb{N}$, $\left\lfloor \frac{2n-5}{3}\right\rfloor\leq c_2(K_4,n)\leq \left\lfloor \frac{2n-3}{3}\right\rfloor$.
Furthermore, for every $n> 98$,
\[c_2(n, K_4) = \left\lfloor \frac{2n-5}{3}\right\rfloor.\]
\end{theorem}

%The statement of our stability result is somewhat lengthy, and we therefore defer it to Section~\ref{subsection: stability}. Our characterization of the extremal configurations attaining the bounds in Theorem~\ref{theorem: K4 threshold} requires us to consider the cases $n\equiv 0,1,2 \mod{3}$ separately and to introduce some technical notation, so we likewise defer a statement of the precise results to Section~\ref{subsection: K4 threshold}.
%Although the stability approach seems like an overkill to close a gap of one, we believe that our stability result, Theorem~\ref{theorem: K4 stability}, is of independent interest. For example, it shows that every $n$-vertex 3-graph $G$ with $\delta_2(G)\ge (2/3 - c)n$ contains \emph{at most} one vertex that is not covered by any copy of $K_4$, for some absolute constant $c>0$ (see Section~\ref{section: conclusion}). Such a result is not true for other graphs, see Remark~\ref{remark: c5 unstable}.

% YZ removed "the generalized triangle" as we never mention it else where and add the commented sentence back as we mentioned this in the abstract.
In addition, we determine $c_2(F)$ when $F$ is $K_4^-$, the strong $5$-cycle $C_5$, and $K_5^-$ --- in fact in each case we give upper and lower bounds on $c_2(n, F)$ differing by at most $2$.
%In addition we show $c_2(K_4^-)=\frac{1}{3}$, $c(C_5)=\frac{1}{2}$ and $c_2(K_5^-)=\frac{2}{3}$, giving in each case upper and lower bounds differing by at most $2$. 
\begin{theorem}\label{theorem: k4-}
	Suppose $n=6m + r$ for some $r \in\{0,1,2,3,4,5\}$ and $m\in \mathbb{N}$, with $n\geq 7$. Then 
	\begin{align*}
	c_2(n, K_4^-)&=\left\{\begin{array}{ll}
	2m-1 \textrm{ or } 2m& \textrm{if }r=0\\
	2m & \textrm{if }r \in \{1,2\}\\
	2m \textrm{ or } 2m+1& \textrm{if }r\in \{3,4\}\\
	2m+1& \textrm{if }r=5.
	\end{array}\right.
	\end{align*}
	In particular, $c_2(K_4^-)=\frac{1}{3}$.
\end{theorem}
\begin{theorem}\label{theorem:  c5}
	$ \lfloor \frac{n-3}{2}\rfloor \leq c_2(n, C_5)\leq \lfloor \frac{n}{2}\rfloor$. In particular, $c_2(C_5)=\frac{1}{2}$.
\end{theorem}
Interestingly, there is no unique stable near-extremal configuration for Theorem~\ref{theorem:  c5}: at least two configurations at edit distance $\Omega(n^3)$ of each other exist, see Remark~\ref{remark: c5 unstable}.
\begin{theorem}\label{theorem: k5-}
	$\left\lfloor \frac{2n-5}{3}\right\rfloor \leq c_2(n, K_5^-)\leq \left\lfloor\frac{2n-2}{3}\right\rfloor$. In particular, $c_2(K_5^-)= \frac{2}{3}$.
\end{theorem}

Let us compare the Tur\'an density $\pi$, the (existence) codegree density $\gamma$, the covering codegree density $c_2$, and the tiling codegree density $t_2$ of $K_4$, $K_4^-$, and $C_5$ in the following table (for a 3-graph $F$ of order $f$, define $t_2(F)=\lim_{n= mf \rightarrow \infty}t_2(n, F)/(n-2)$ if this limit exists). In the table question marks indicate conjectures, except for $t_2(C_5)$, for which we are not aware of any conjecture.

\begin{center}
\medskip
{\renewcommand{\arraystretch}{1.5}
\renewcommand{\tabcolsep}{0.2cm}
\begin{tabular}{| c | c | c | c | c |}
\hline
	& $\gamma$ & $\pi$ & $c_2$ & $t_2$\\ \hline
$K_4$ & $\frac12$? \cite{CzygrinowNagle01} & $\frac59$? \cite{Turan41} & $\frac23$ & $\frac34$ \cite{KeevashMycroft14, LoMarkstrom11}\\ \hline
 $K_4^-$ & $\frac14$ \cite{FalgasRavryPikhurkoVaughan15+} &$\frac27$? \cite{FranklFuredi84} &  $\frac13$ & $\frac12$ \cite{LoMarkstrom13} \\ \hline
$C_5$ & $\frac13$? \cite{Marchant11} & $2\sqrt{3} - 3$? \cite{MubayiRodl02} & $\frac12$ & ? \\
\hline 
\end{tabular} }
\medskip
\end{center}
Finally we give bounds on $c_2(\mathrm{Fano})$, $c_2(F_{3,2})$ and $c_2(K_t)$ for $t\ge 5$, and pose a number of questions.

Our paper is structured as follows. %in Section~\ref{section: preliminaries} we prove the existence of $c_2(F)$ and some basic inequalities for $c_2(n,F)$, as well as a general upper bound on $c_2(n,F)$. 
In Section~\ref{section: k4 results}, we determine the codegree covering threshold for $K_4$ and characterize the extremal configurations. 
In Section~\ref{section: other 3-graphs}, we prove our bounds on $c_2(n,F)$ for the other $3$-graphs $F$ mentioned above. We end in Section~\ref{section: conclusion} with some discussion and questions.

\section{The covering codegree threshold for $K_4$}\label{section: k4 results}
In this section we determine the codegree threshold $c_2(n, K_4)$. We give a lower bound construction in Section~\ref{subsection: k4 covering codegree density} and prove the upper bound in Section~\ref{sec:K4upper}. Finally, in Section~\ref{subsection: K4 extremal constructions} we provide other extremal constructions, and state a stability theorem that helps to show that these constructions are all possible extremal configurations; as the proofs of these latter results are similar to the proof of Theorem~\ref{theorem: K4 threshold} %we omit them here (we refer an interested reader to the appendix to the arxiv version of this paper~\cite{FalgasRavryZhao15+} for the details).
we defer them to the appendix. 
 
%and the extremal configurations for $n> 98$. Our main tool is a strong stability result proved in Section~\ref{subsection: stability}. We then describe the set of extremal configurations in Section~\ref{subsection: K4 extremal construcitons} and finish in Section~\ref{subsection: K4 threshold} with a proof of Theorem~\ref{theorem: K4 threshold}.

\subsection{Lower bound}\label{subsection: k4 covering codegree density}
\begin{proof}[Proof of the lower bound in Theorem~\ref{theorem: K4 threshold}]
	We construct a $3$-graph $F_1(n)$ on $V=[n]$. Select a special vertex $x$. Split the remainder of the vertices into three parts $V_1\sqcup V_2 \sqcup V_3=V\setminus\{x\}$ with sizes as equal as possible, 
	\[\vert V_3\vert -1\leq \vert V_1\vert \leq \vert V_2\vert \leq \vert V_3\vert.\] 
	Put in as the link graph of $x$ all pairs between distinct parts, i.e. add in all triples of the form $xV_iV_j$ for $i\neq j$. Further, add in all triples not containing $x$ and meeting at most two of the three parts $(V_i)_{i=1}^3$. Denote the resulting $3$-graph by $F_1=F_1(n)$. The complement of $F_1(n)$ is shown in Figure~1. %\ref{figure: k4 construction}.

	\begin{figure}\label{figure: k4 construction}
		
		\centering
		\begin{tikzpicture}
		
		\draw (0,0) circle [radius=1];
		\draw (3,0) circle [radius=1];
		\coordinate (gamma) at ($ (0,0)!1!60:(3,0) $);
		\draw (gamma) circle [blue, radius=1];
		
		\coordinate (delta) at ($ (0,0)!1!30:(1.73,0) $);
		\fill[red] (delta) circle (2pt) node[above] {$x$};

		\coordinate (a) at ($ (0,0)!1!30:(0.25,0) $);
		\coordinate (b) at ($ (3,0)!1!150:(3.25,0) $);
		\coordinate (c) at ($ (a)!1!60:(b) $);
		\draw[blue!100, thick, fill=blue!10] (a)-- (b) -- (c) -- (a);
		
		\coordinate (a1) at (-0.4, 0);
		\coordinate (a2) at (0, -0.4);
		\draw[thick, red!100] (a1) -- (a2); 
		
		\coordinate (b1) at (3.4, 0);
		\coordinate (b2) at (3, -0.4);
		\draw[thick, red!100] (b1) -- (b2);
		
		\coordinate (c1) at ($(gamma)+(-0.35,0.35)$);
		\coordinate (c2) at ($(gamma)+(0.35,0.35)$);
		\draw[thick, red!100] (c1) -- (c2);

		\node (aedge) at (-0.35,-0.35)[red] {$x$};
		\node (bedge) at (3.37 , -0.37)[red] {$x$};
		\node (cedge) at ($(gamma)+(0,0.5)$)[red] {$x$};

		%\fill[black!10] (a') -- (b')  -- (c') --(a');
		\node (cn) at ($(gamma)+(0,1.5)$)[] {$V_3$};
		\node (an) at (0, -1.5)[] {$V_1$};
		\node (bn) at (3, -1.5)[] {$V_2$};
		\end{tikzpicture}
		\caption{The complement of $F_1(n)$. The red pairs and the blue triples are absent from the link graph of $x$ in $F_1$ and from $E(F_1)$ respectively.}
	\end{figure}
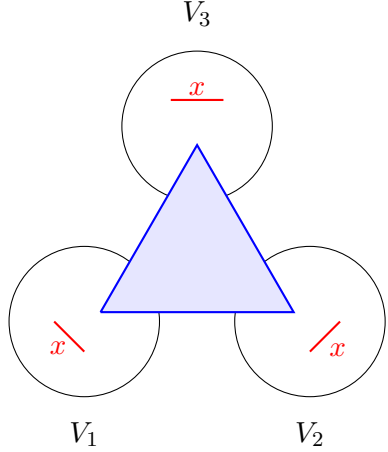

	Observe that $x$ is contained in no copy of $K_4$ in $F_1$: the only triangles in the link graph of $x$ are tripartite, and thus are not covered by any triple of the $3$-graph. Let us now compute the minimum codegree of $F_1$. For $v_i, v_i' \in V_i$ and $v_{i+1}\in V_{i+1}$, we have $d(v_i,x)= n-1-\vert V_i\vert$, $d(v_i,v_i')=n-3$ and $d(v_i, v_{i+1})=n-2-\vert V_{i+2}\vert $. The minimum codegree $\delta_2(F_1)$ is thus $n-2-\lceil \frac{n-1}{3}\rceil$, attained by pairs $(v_1,v_2)\in V_1\times V_2$. Writing $n=3m+r$ with $r \in\{0,1,2\}$ and $m \in \mathbb{N}$, we have shown that 
	\[c_2(3m+r, K_4)\geq \delta_2(F_1) = \left\{\begin{array}{ll}
	2m-2 & \textrm{if $r=0$}\\
	2m-1 & \textrm{if $r=1$ or $2$.}
	\end{array} \right.\]
	This lower bound can be expressed more compactly as $c_2(n, K_4)\geq \lfloor \frac{2n-5}{3} \rfloor$.
\end{proof}

%\subsection{A first upper bound}\label{subsection: K4 first upper bound}
%Let $G$ be a %$3$-graph on $n\geq 4$ vertices with $\delta_2(G)> \frac{2n-3}{3}$. Pick an arbitrary $x\in V(G)$. Let $abx$ be any $3$-edge containing $x$. We have $d(a,b)+d(a,x)+d(b,x)-3>2(n-3)$. So by the pigeonhole principle, there exists at least one vertex $c \in V(G)\setminus \{a,b,c\}$ which makes a $3$-edge of $G$ with at least three of $ab$, $ax$, $bx$. Thus the $4$-set $abcx$ induces a copy of $K_4$ in $G$ covering $x$, as required. Setting $n=3m+\epsilon$ for $\epsilon \in\{0,1,2\}$ and $m\in \mathbb{N}$, this shows that 
%\[c_2(3m+\epsilon, K_4)\leq \left\{\begin{array}{ll}
%2m-1 & \textrm{if $\epsilon=0$ or $1$}\\
%2m & \textrm{if $\epsilon=2$}
%\end{array} \right.\]
%\end{proof}
%\begin{remark} In the case $\epsilon=1$, this upper bound matches the lower-bound on $c_2(n, K_4)$ proved in the previous subsection. In the cases $\epsilon=0$ or $2$, however, a gap of $1$ remains between our upper and lower bounds. We shall use the stability results proved in the next subsection to remedy this.
%\end{remark}
%\begin{corollary}
%	$c_2(K_4)=\frac{2}{3}$.\qed
%\end{corollary}

%I think this may be tight -- in the sense that we can tweak some of our construction and get something where every vertex is in a K4 but there are some 3-edges that cannot be extended to a K4
%To sharpen the upper bound on $c_2(n,K_4)$ and characterize the extremal constructions, we need a stability result which is the subject of Section~\ref{section: k4 results}.

\subsection{Upper bound}
\label{sec:K4upper}

Let us give a general upper bound for $c_2(n,F)$, which turns out to be surprisingly close to the truth in the case of $F=K_4$. %YZ deleted "but quite far away in other cases." because the lemma is in fact tight for some other cases.

\begin{lemma}\label{lemma: trivial upper bound}
	Given a $3$-graph $F$ with at least one $3$-edge, let $r$ be the maximum of $\delta_1(F')$ among all subgraphs $F'$ of $F$. Then $c_2(n, F)\le \lfloor (1- 1/r)n + (\vert V(F)\vert - 2r - 1)/r \rfloor$.
\end{lemma}

\begin{proof}
	Assume that $F$ contains $f$ vertices. We order the vertices of $F$ as $x_1, \dots, x_f$ such that $x_i$ is a vertex of minimum degree in the subgraph $F - \{x_{i+1}, \dots, x_f\}$. As $r= \max \delta_1(F')$ among all subgraphs $F'$ of $F$, we know that $x_i$ has at most $r$ neighbours among $x_1, \dots, x_i$. 
	
	Let $G$ be a 3-graph on $n$ vertices such that 
	\[
	\delta_2: = \delta_2(G)> \frac{r-1}{r} n + \frac{f-1}{r} - 2.
	\]
	Fix a vertex $v_1$ of $G$. We will find a copy of $F$ in $G$ by first mapping $x_1$ to $v_1$, $x_2$ to any other vertex $v_2$, and $x_3$ to any $v_3\in \Gamma_G(v_1, v_2)$.
	% YZ says mapping x_3 to a neighbour of v_1 and v_2, which matches what we claim in the remark.  
	Suppose that $x_1, \dots, x_i$ have been embedded to $v_1, \dots, v_i$. In order to embed $x_{i+1}$, we consider the neighbours of $x_{i+1}$ among $x_1, \dots, x_i$. There are $t\le r$ such neighbours and they are mapped to pairs $p_1, \dots p_t$ of $v_1, \dots, v_i$. Each $p_j$ has at least $\delta_2$ neighbours in $G$ and thus at most $n-2 - \delta_2$ non-neighbours in $V(G) \setminus \{v_1, \dots, v_i\}$.
	By the definition of $\delta_2$ and $i\le f-1$, we have $r(n-2- \delta_2)< n - i$. Hence there exists a vertex $v_{i+1}\in V(G) \setminus \{v_1, \dots, v_i\}$ such that $v_{i+1}$ is a common neighbour of $p_1, \dots p_t$.  Continuing this process, we obtain a copy of $F$ as desired. 
\end{proof}
\begin{remark}
	The proof of Lemma~\ref{lemma: trivial upper bound} actually shows that if $\delta_2(G)> (1- 1/r)n + (\vert V(F)\vert - 2r - 1)/r $ then every \emph{triple} of $E(G)$ is covered by an $F$-subgraph.
\end{remark}

Applying Lemma~\ref{lemma: trivial upper bound} with $F=K_4$ and $r=3$, we obtain that $c_2(n, K_4) \leq \lfloor \frac{2n-3}{3}\rfloor$. When $n \equiv 1 \mod 3$, this implies that $c_2(n, K_4) \leq \lfloor \frac{2n-5}{3}\rfloor$. Together with the lower bound $c_2(n, K_4) \ge \lfloor \frac{2n-5}{3}\rfloor$, we obtain $c_2(n, K_4) = \lfloor \frac{2n-5}{3}\rfloor$ immediately.

When $n \equiv 0$ or $2 \mod 3$,  more work is required to reduce the upper bound to $\lfloor \frac{2n-5}{3}\rfloor$. In both cases, we shall make use of the following simple observation. 

\begin{lemma} \label{lem:Sy}
Let $G$ be a 3-graph on $n\ge 4$ vertices. Suppose that $x\in V(G)$ is not covered by any copy of $K_4$ and there exists $a, b, c\in V(G)$ such that $abx, bcx, acx\in E(G)$ (thus $abc\not\in E(G)$). Let $S=\{a, b, c, x\}$ and for each vertex $y \in V(G)\setminus S$, let $S_y$ consist of all the pairs of $S$ that make a $3$-edge with $y$ in $G$. Then $S_y$ must be a subset of one of the following sets: 
\[
S^{1, c}=\{ax,bx, ac, bc\}, \quad S^{1,b}= \{ax,cx, ab, bc\}, \quad S^{1,a}=\{bx, cx, ab, ac\}, 
\]
\[
S^{2,a}=\{ab,ac, bc, ax\}, \ S^{2,b}=\{ab,ac, bc, bx\}, \ S^{2,c}=\{ab,ac, bc, cx\}, \  S^3=\{ax,bx,cx\}.
\]
In particular, $|S_y|\le 4$. \qed
\end{lemma}

\begin{proof}[Proof of $c_2(n, K_4) \leq \lfloor (2n-5)/{3}\rfloor$ when $3$ divides $n$]

Suppose $n=3m$ for some integer $m\ge 2$ (so that $\lfloor (2n-5)/{3}\rfloor = 2n/3 - 2$). Let $G=(V, E)$ be a 3-graph on $n$ vertices with $\delta_2(G)\ge 2n/3 - 1$. We claim that all vertices of $G$ are covered by copies of $K_4$.
%Trivially the 3-graph consisting of three vertices and a 3-edge contains no $K_4$-covering. Hence $c_2(3, K_4) = 1$. We thus assume that $n\ge 4$.
Suppose instead, that some vertex $x\in V$ is not contained in a copy of $K_4$. Since the minimum degree in the link graph $G_x$ of $x$ is at least $2n/3 -1> (n-1)/2$, there exists a triangle $\{ab, bc, ac\}$ in $G_x$. This implies that $abc\notin E$. Set $S=\{a,b,c,x\}$. For each vertex $y \in V\setminus S$, by Lemma~\ref{lem:Sy}, at most four pairs of $S$ form edges of $G$ with $y$. 
Thus, by the codegree assumption, 
\[
6 \left(\frac{2n}{3} - 1 \right) \le d(a,x)+ d(b,x)+ d(c,x)+d(a,b)+d(b,c)+d(c,a) \le 4(n-4) +9,
\]
a contradiction. 
\end{proof}

\medskip
When $n \equiv 2 \mod 3$, we start the proof in the same way. However, since we only have $\delta_2(G)\ge (2n-4)/3$, we will not obtain a contradiction until we prove that $G$ has a similar structure as the 3-graph $F_1(n)$ given in Section~\ref{subsection: k4 covering codegree density}.

\begin{proof}[Proof of $c_2(n, K_4) \leq \lfloor \frac{2n-5}{3}\rfloor$ when $n \equiv 2 \mod 3$]

Suppose $n=3m+2> 98$. In order to show that $c_2(n, K_4) \le \lfloor \frac{2n-5}{3}\rfloor = (2n-7)/3$, consider a 3-graph $G=(V, E)$ on $n$ vertices satisfying $\delta_2(G)\ge (2n-4)/3$. 

%
%Assuming Lemma~\ref{lem:Case3}, we can prove  as follows. Assume that $n=3m+2> 170$. Let $G$ be a 3-graph on $n$ vertices with $\delta_2(G)\ge (2n-4)/3=2m$. Let $x, V_1, V_2, V_3$ be defined as in Lemma~\ref{lem:Case3}. Without loss of generality, assume that $|V_1| \le |V_2|\le |V_3|$. Then $|V_3|\ge m+1$. For $v_3\in V_3$, Lemma~\ref{lem:Case3} gives that $\Gamma(x, v_3)\subseteq V_1\cup V_2$. Thus, by
%the codegree condition, 
%	\begin{align*}
%	2m \leq d(x,v_3)\leq \vert V_1\vert +\vert V_2\vert =3m+1-\vert V_3\vert.
%	\end{align*}
%	This implies that $\vert V_3 \vert = m + 1$ and all triples of the form $xv v_3$ with $v_3\in V_3$ and $v\in V_1\cup V_2$ are in $E(G)$. Fix a vertex $v_1\in V_1$. Since $\Gamma(x, v_1)\subseteq V_2\cup V_3$ and $d(x, v_1) > |V_3|$, there exists $v_2\in V_2$ such that $xv_1v_2\in E(G)$. The triple $v_1v_2v_3$ does not lie in $E(G)$ for any $v_3\in V_3$, since otherwise $xv_1v_2v_3$ would induce a copy of $K_4$. We thus have 
%	\[ 2m\le d(v_1,v_2)\leq \vert V_1 \vert +\vert V_2\vert -1= 2m-1, 
%	\]
%	a contradiction. 
%	
%
%\begin{proof}[Proof of Lemma~\ref{lem:Case3}]

Suppose that a vertex $x$ of $G$ is not contained in any copy of $K_4$. As $ (2n-4)/3 > (n-1)/2$, the link graph $G_x$ contains a triangle $\{ab, bc, ac\}$. Set $S=\{a,b,c,x\}$ and for each $y\in V\setminus S$, define $S_y$ as in Lemma~\ref{lem:Sy}. By Lemma~\ref{lem:Sy}, $S_y$ is a subset of $S^{1, c}, S^{1,b}$, $S^{1,a}$, $S^{2,a}$, $S^{2,b}$, $S^{2,c}$ or $S^3$. 
	For $i\in\{1,2\}$ and $j\in\{a,b,c\}$, write $s_{i,j}$ for the number of vertices $y \in V\setminus S$ for which $S_y = S^{i,j}$, and write $s_i$ for the sum $s_{i,a}+s_{i,b}+s_{i,c}$. Finally let $s_0$ be the number of vertices $y\in V\setminus S$ such that $S_y \neq S^{i,j}$ for any $i\in \{1,2\}$ and $j\in\{a,b,c\}$. Note that $|S_y|\le 3$ for such $y$. We know that $s_1+s_2+s_0 = n-4$. Furthermore, by the codegree assumption, 
	\begin{equation}\label{ax-bx-cx}
	3\, \frac{2n- 4}{3}  \leq d(a,x)+ d(b,x)+ d(c,x)\leq 2s_1+s_2+3s_0 +6,
	\end{equation}
	\begin{equation}\label{ax-bx-cx-ab-bc-ac}
	6\, \frac{2n- 4}{3} \leq d(a,x)+ d(b,x)+ d(c,x)+d(a,b)+d(b,c)+d(c,a) \leq 4s_1+4s_2 +3s_0 +9,
	\end{equation}
	Substituting $s_0=n-4-s_1-s_2$ into (\ref{ax-bx-cx}) and  (\ref{ax-bx-cx-ab-bc-ac}) yields that $s_1+2s_2 \le n-2$ and $s_1+s_2  \geq n-5$, respectively. Combining the two inequalities we have just obtained, we get
	\begin{align*}
	s_2 \leq 3 \quad \text{and} \quad s_1 \geq n- 8.
	\end{align*}
	We now show that the weight of $s_1$ splits almost equally between $s_{1,a}$, $s_{1,b}$, $s_{1,c}$. Note that
	\begin{align*}
	\frac{2n- 4}{3} \leq d(b,c)\leq n-3-s_{1,a}, % YZ changed -2 to -3 because $abc\not\in E(G)$
	\end{align*}
	from which it follows that $s_{1,a}\leq \frac{n-5}{3}$. Similarly we derive that $s_{1,b,} s_{1,c}\le (n-5)/3$. Consequently
	\begin{align*} 
	s_{1,a}=s_1-s_{1,b}-s_{1,c} &\geq n- 8 -2\frac{n-5}{3} =\frac{n - 14}{3}.
	\end{align*}
	Similarly $s_{1,b}$ and $s_{1,c}$ satisfy the same lower bound. Let $A=\{y\in V\setminus S:\ S_y= S^{1,a}\}\cup \{a\}$, $B=\{y\in V\setminus S:\ S_y= S^{1,b}\}\cup\{b\}$ and $C=\{y\in V\setminus S:\ S_y= S^{1,c}\}\cup \{c\}$.  Set $V'=A\cup B\cup C\cup \{x\}$. Then we have just shown the following lemma.
	\begin{lemma}\label{lemma: A,B,C cover almost everything}
		\[\vert V'\vert= 1+\vert A\vert + \vert B\vert + \vert C\vert \geq n-4, \quad \text{and} \quad 
		\frac{n - 11}{3} \le \vert A\vert, \vert B\vert, \vert C\vert \le \frac{n-2}{3}. \qed
		\]
	\end{lemma} 
	%\begin{proof}
	%By the codegree assumption,
	%\begin{align*}
	%(4-6\delta)n & \leq d(a,x)+d(b,x)+d(c,x)+d(a,b)+d(b,c)+d(a,c)\\
	%&\leq  6+4 \left(\vert A\vert +\vert B\vert + \vert C\vert +s_2\right)+ 3(n-4 -\vert A\vert -\vert B\vert -\vert C\vert -s_2)\\
	%&= 3n +\vert A\vert +\vert B\vert + \vert C\vert +s_2 -6
	%\end{align*}
	%whence
	%\begin{align*}
	%\vert A\vert +\vert B\vert + \vert C\vert & \geq (1-6\delta)n - s_2 -6 \geq (1-15\delta)n +6.
	%\end{align*}
	%For the second part of the claim, note that
	%\[\vert V'\vert -s_{1,b}-s_{1,c} \leq \vert A\vert \leq s_{1,a}.\]
	%The bounds on the size of $\vert A\vert$ then follow from our bound on $\vert V'\vert$ above and our previous bounds on $s_{1,a}$, $s_{1,b}$ and $s_{1,c}$. The bounds on $\vert B\vert$ and $\vert C\vert $ follow on by symmetry.
	%\end{proof}
%	Given three pairwise disjoint sets $V_1,V_2,V_3$ and a vertex $x\notin \bigcup_i V_i$, we define a $3$-graph $F_1(V_1,V_2,V_3,x)$ on $V_1\cup V_2\cup V_3 \cup \{x\}$ by adding all triples of the form $xV_iV_{j}$ and $V_iV_iV_j$ whenever $i\ne j$. Let $G' = G[V']$ be the restriction of $G$ to $V'$. 
%	We now show that $G'$ is close to $F_1=F_1(A,B,C, x)$. 

	Let $\mathcal{B}$ be the collection of $3$-edges of $G$ of the form $xAA, xBB, xCC$ (the `bad' triples). Let $\mathcal{M}$ be the 
collection of non-edges of $G$ of the form $xAB, xAC, xBC$ (the `missing' triples). Viewing $\mathcal{B}$ and $\mathcal{M}$ as $3$-graphs on $V'$, for two distinct vertices $v_1, v_2\in V'$, we let $d_{\mathcal{B}}(v_1,v_2)$ denote their codegree in $\mathcal{B}$ and $d_{\mathcal{M}}(v_1,v_2)$ their codegree in $\mathcal{M}$.
	
	\begin{claim}\label{claim: all bad vx degrees small}
		For every $v\in V'\setminus \{x\}$, $d_{\mathcal{B}}(v,x) \leq 4$. 	
	\end{claim}
	\begin{proof}
		Suppose without loss of generality that $v\in A$.  If $v=a$, then $d_{\mathcal{B}}(v,x)=0$ because $G$ contains no 3-edges of the form $x a A$. We thus assume that $v\ne a$. The bad triples for the pair $(v,x)$ are triples of the form $a'vx$ for $a'\in A\setminus \{a,v\}$. Suppose $a'vx\in \mathcal{B}$. Then since there is no $K_4$ in $G$ containing $x$, and since, by the definition of $A$, $a'bx$, $vbx$, $a'cx$ and $vcx$ are all in $G$, it must be the case that both of $a'vb$ and $a'vc$ are missing from $G$. 
		Further if $c'\in C\cap \Gamma(v,x)$ then all of $c'vx, bvx, c'bx$ are in $G$, whence $bc'v$ is absent from $G$. Similarly for any $b'\in B$, at most one of $b'cv$, $b'xv$ is in $G$. Finally since $bc v\not\in E(G)$, $b$ and $c$ are contained in exactly one of $\Gamma(b, v)$, $\Gamma(c, v)$, and $\Gamma(x, v)$. To summarize, a vertex $y$ in $V'$ can lie in at most two of $\Gamma(b,v)$, $\Gamma(c,v)$ and $\Gamma(x,v)$ unless $y$ is in $\Gamma_{\mathcal{B}}(x,v)$ (and lies in exactly one of those joint neighbourhoods) or is in $\{b,c,v\}$ (and lies in at most one of those joint neighbourhoods).
			% YZ replaced "Further if $c'\in C\cap \Gamma(a',x)$ then all of $a'c'x, a'bx, a'cx$ are in $G$, whence $a'bc'$ is absent from $G$. Similarly for any $b'\in B$, at most one of $a'b'c$, $a'b'x$ is in $G$."
		Together with our codegree assumption, this gives us
		\begin{align*}
			3\, \frac{2n-4}{3} \leq d(b,v)+ d(c,v)+d(x,v) &\leq 2\vert V'\vert-d_{\mathcal{B}}(v,x) -4+3(n-\vert V'\vert)\\
			&= 3n -\vert V'\vert- 4 -d_{\mathcal{B}}(v,x) \leq 2n - d_{\mathcal{B}}(v,x),	\end{align*}
		where we apply $\vert V'\vert\ge n -4$ from Lemma~\ref{lemma: A,B,C cover almost everything}
 in the last inequality. It follows that $d_{\mathcal{B}}(v,x)\leq 4$, as claimed.		
	\end{proof}
	
	\begin{claim}\label{claim: all missing vx degrees small}
		For every $v\in V'\setminus \{x\}$, $d_{\mathcal{M}}(v,x) \leq 8$. 	
	\end{claim}
	\begin{proof}
		Suppose without loss of generality that $v\in A$. Then by the codegree assumption, Claim~\ref{claim: all bad vx degrees small} and the bound on $\vert A\vert$ from Lemma~\ref{lemma: A,B,C cover almost everything} we have
		\begin{align*}
		\frac{2n- 4}{3} \leq d(v,x)&\leq n-1-\vert A\vert +d_{\mathcal{B}}(v,x)-d_{\mathcal{M}}(v,x)
		\leq n-1 - \frac{n-11}3 +4 - d_{\mathcal{M}}(v,x),
		\end{align*}
		which gives that $d_{\mathcal{M}}(v,x)\leq 8$ as claimed.
	\end{proof}
	
	\begin{claim}\label{claim: no bad vx triples}
		For every $y \in V(G)\setminus \{x\}$, $\Gamma(y,x)$ has a non-empty intersection with at most two of the parts $A$, $B$ and $C$.
	\end{claim}
	\begin{proof}
		Let $y \in V(G)\setminus \{x\}$. Set $A_y=A\cap\Gamma(x,y)$, $B_y=B\cap \Gamma(x,y)$ and $C_y=C\cap 	\Gamma(x,y)$.  Suppose none of $A_y$, $B_y$, $C_y$ is empty. Fix $a'\in A_y$. For $b'\in B_y$, if $b'\in 		\Gamma(a', x)$, then $a' b' y\not\in E(G)$ -- otherwise $\{a', b', x, y\}$ spans a copy of $K_4$. Similarly, for $c'\in 	C_y\cap \Gamma(a', x)$, we have $a' c' y\not\in E(G)$. Hence, 
		\begin{align*}
		\frac{2n-4}{3} \leq d(a',y) &\leq  n-2 - | B_y \cap \Gamma(a', x)| - | C_y \cap \Gamma(a', x) |.
		\end{align*}		
	Claim~\ref{claim: all missing vx degrees small} gives that $d_{\mathcal{M}}(a',x)\leq 8$.  Consequently,
		\[
		| B_y \cap \Gamma(a', x)| + | C_y \cap \Gamma(a', x) | = |B_y| + |C_y| - d_{\mathcal{M}}(a',x) \ge |B_y| + |C_y| - 8
		\]
		This 	implies that	
		\[
		\frac{2n-4}{3} \leq  n-2 -\vert B_y\vert - \vert C_y\vert + 8, 
		\]
		which  yields $\vert B_y\vert +\vert C_y\vert \leq (n+22)/{3}$. Similarly by considering any vertex $b'\in B_y$ and any vertex $c'\in C_y$ we obtain that 
		\begin{align*}
		\vert A_y \vert +\vert C_y\vert \leq \frac{n+22}{3} \quad  \text{and} \quad \vert A_y \vert +\vert B_y\vert \leq \frac{n+22}{3}.
		\end{align*}
		Summing these three inequalities and dividing by $2$, we obtain that
		\[
		\vert A_y \vert + \vert B_y\vert + \vert C_y \vert \leq \frac{n+22}{2}.
		\]
		Furthermore, by the codegree condition, 
		\[
		\frac{2n-4}{3} \leq d(x,y) \leq \vert A_y\vert+\vert B_y\vert + \vert C_y\vert  + \left(n-\vert V'\vert\right)
		\leq \frac{n+22}{2} +4,
		\]
		where we apply $\vert V'\vert\ge n -4$ from Lemma~\ref{lemma: A,B,C cover almost everything}. Rearranging terms yields $\frac{n}6 \le \frac{49}3$, which contradicts our assumption $n> 98$. 
	\end{proof}
	Set $V_1= \{y \in V\setminus \{x\}: \ \Gamma(x,y)\cap A=\emptyset \}$, $V_2=\{y \in V\setminus \{x\}: \ \Gamma(x,y)\cap B=\emptyset \}$ and $V_3 =\{y \in V\setminus \{x\}: \ \Gamma(x,y)\cap C=\emptyset \}$. 
	Without loss of generality, assume that 
	\begin{equation} \label{eq:Vi}
	|V_1| \le |V_2|\le |V_3|.
	\end{equation}
Claim~\ref{claim: no bad vx triples} shows that $V_1\cup V_2\cup V_3$ covers $V(G)\setminus \{x\}$. We now show that in fact $V_1, V_2, V_3$ are pairwise disjoint, and $A\subseteq V_1$, $B\subseteq V_2$, and $C\subseteq V_3$.  
Suppose instead, that there exists $y\in V_1\cap V_2$. Then $\Gamma(x,y)\cap (A\cup B) =\emptyset$. By the codegree condition and Lemma~\ref{lemma: A,B,C cover almost everything}, 
\[
\frac{2n-4}{3} \leq d(x,y) \leq \vert C_y\vert  + \left(n-\vert V'\vert\right) \le \frac{n-2}{3} + 4,
\]	
	which implies that $n\le 14$, a contradiction.
	 
	Furthermore, consider $a'\in A$. By Claim~\ref{claim: all missing vx degrees small}, $a'xv\in E(G)$ for all but at most $8$ vertices $v\in B\cup C$. By Lemma~\ref{lemma: A,B,C cover almost everything},
	\begin{align*}
	\vert B \vert -8 \geq \frac{n-11}3 -8 > 0
	\end{align*}
	which is strictly positive as $n>35$. Thus we have that $\Gamma(a',x)$ has a non-empty intersection with $B$; similarly we have that $\Gamma(a',x)\cap C\neq \emptyset$, from which we can finally deduce by Claim~\ref{claim: no bad vx triples} that $\Gamma(a',x)\cap A=\emptyset$ and that $A\subseteq V_1$. Similarly we have $B\subseteq V_2$ and $C\subseteq V_3$. %This implies that there exists $v_1\in V_1$ and $v_2\in V_2$ such that $x v_1 v_2\in E(G)$.
	
	Let $c' \in C$. By the definition of $V_3$, we have $\Gamma(c', x)\subseteq V_1\cup V_2$. By the codegree assumption, it follows that
	\begin{align} \label{eq:2m}
	\frac{2n-4}{3} \leq d(c', x)& \leq \vert V_1\vert +\vert V_2\vert = n-1 -\vert V_3\vert, 
	\end{align}
	from which we get that $\vert V_3\vert \leq (n+1)/3 $. Since $n=3m+2$, by \eqref{eq:Vi}, we derive that $|V_3|= (n+1)/3= m+1$ and $|V_1|\le |V_2|\le (n+1)/3$.
	%Consequently, $|V_1|+ |V_2|\le 2m$.
	
	\begin{claim}\label{claim: all missing x degrees small}
		Let $y\in V_i$. Then $\Gamma(y, x)$ contains all but at most $6$  vertices from $\bigcup_{j\ne i} V_j$ and no vertex from $V_{i}$. 
		%$V_{i+1}\cup V_{i+2}$  (where the indices wind round modulo $3$).
	\end{claim}
	\begin{proof}
		Suppose without loss of generality that $y\in V_1$. Then by Claim~\ref{claim: no bad vx triples}, $A\cap \Gamma(y,x)=\emptyset$. Thus
		\begin{align*}
		\frac{2n - 4}{3} \leq d(x,y)\leq \vert \Gamma(x,y)\cap \left(V_2\cup V_3\right)\vert + \vert \Gamma(x,y) \cap \left(V_1\setminus A\right)\vert
		\leq \vert \Gamma(x,y)\cap \left(V_2\cup V_3\right)\vert + 4
		\end{align*}
		because $\vert V_1\setminus A \vert \le n - \vert V'\vert\le 4$ by Lemma~\ref{lemma: A,B,C cover almost everything}. Hence $\vert \Gamma(x,y)\cap \left(V_2\cup V_3\right)\vert \ge (2n - 16)/3$. Since $|V_i|\le (n+1)/3$ for all $i$,
		\[
		| \left(V_2\cup V_3\right) \setminus \Gamma(x,y) | \le 2\,\frac{n+1}3 - \frac{2n-16}3 =6.
		\]
		This establishes the first part of our claim. 
				
		For the second part of our claim (namely, $\Gamma(y,x)\cap V_1=\emptyset$), suppose that  $yy'x\in E(G)$ for some $y'\in V_1$. Then $\Gamma(y, y') \cap \Gamma(y,x)\cap \Gamma(y',x) = \emptyset$. Consequently,
		\begin{align*}
		\frac{2n - 4}{3} \leq d(y, y') \leq 1 + \vert V_1\vert-2 + \left\vert\left( V_2\cup V_3\right)\setminus\left(\Gamma(y,x)\cap \Gamma(y',x)\right)\right\vert \leq 1+ \frac{n+1}{3} -2+ 2 \cdot 6 
		\end{align*}
		where in the last inequality we apply $|V_1| \le (n+1)/3 $ and the first part of the claim. This implies that $n\le 38$, a contradiction.
	\end{proof} 
%Fix a vertex $v_1\in V_1$. Since $\Gamma(x, v_1)\subseteq V_2\cup V_3$ and $d(x, v_1) > |V_3|$, 

Claim~\ref{claim: all missing x degrees small} implies that $\Gamma(v_3, x)\subseteq V_1\cup V_2$ for all $v_3\in V_3$. Then $d(v_3, v)$ satisfies \eqref{eq:2m} with two inequalities replaced by equalities. Consequently
all triples of the form $xv v_3$ with $v_3\in V_3$ and $v\in V_1\cup V_2$ are in $E(G)$

Claim~\ref{claim: all missing x degrees small} also implies that most $v_1\in V_1$ and $v_2\in V_2$ satisfy $xv_1v_2\in E(G)$. Fix such $v_1$ and $v_2$. Then $v_1v_2v_3\not\in E(G)$ for any $v_3\in V_3$ otherwise $xv_1v_2v_3$ induces a copy of $K_4$. We thus have 
	\[ 2m\le d(v_1,v_2)\leq \vert V_1 \vert +\vert V_2\vert -1= 2m-1, 
	\]
	a contradiction. This completes the proof of Theorem~\ref{theorem: K4 threshold} in the case $n=3m+2$.
\end{proof}
	
\subsection{Other extremal constructions and stability}
\label{subsection: K4 extremal constructions}

Recall the construction $F_1(n)$ described in Section~\ref{subsection: k4 covering codegree density}. 
There are other extremal families of $3$-graphs for $K_4$-covering that are not isomorphic to subgraphs of $F_1(n)$ . 

%To describe these families it is helpful to introduce some notation. Given a tripartition $V_1\sqcup V_2\sqcup V_2= V\setminus \{x\}$, call triples of the form $V_1V_2V_3$ in $F_1$ \emph{tripartite} triples. Also call a pair of vertices $uv$ coming from different parts of the tripartition a \emph{bipartite} pair, and refer to the collection of tripartite triples containing $uv$ as the tripartite triples \emph{supported} by $uv$.
%
%The basic idea behind the alternative constructions is to delete some triples $xuv$ from $F_1$ and to add in tripartite triples supported by the corresponding bipartite pairs $uv$. Observe that such operations do not create a $K_4$ covering $x$: if three vertices $a,b,c$ induce a triangle in the link graph of $x$, then $a,b,c$ come from distinct parts, and, since all of $xab,xac, xbc$ are in the $3$-graph, the tripartite triple $abc$ is absent. What is more, the only pairs of vertices whose codegree decreases when we perform such operations are the pairs $(u,x)$ for which some triple $xuv$ has been deleted. 

\noindent\textbf{Case 1: $n=3m$.}  We partition $[n]\setminus \{x\}$ into three parts $V_1$, $V_2$ and $V_3$ with sizes $\vert V_1\vert= m-1$ and $\vert V_2\vert=\vert V_3\vert=m$. A collection $\mathcal{E}$ of pairs of vertices from different parts of $[n]\setminus\{x\}$ is called \emph{admissible} if (i) every vertex $v_1\in V_1$ is contained in at most two pairs from $\mathcal{E}$, and (ii) every vertex $v \in V_2 \sqcup V_3$ is contained in at most one pair from $\mathcal{E}$. %We call such a collection $\mathcal{E}$ \emph{admissible}. 
Now let $F_1(\mathcal{E}, 3m)$ be the $3$-graph obtained from $F_1$ by deleting all triples $xuv$ and adding all tripartite triples $uvw$ (namely, $w\in V\setminus \{x\}$ is from the part different from the ones containing $u$ or $v$) for all $uv \in \mathcal{E}$. It is easy to see that $F_1(\mathcal{E}, 3m)$ contains no $K_4$ covering $x$ and $\delta_2(F_1(\mathcal{E}, 3m))=\delta_2(F_1(3m))=2m-2$.

\noindent\textbf{Case 2: $n=3m+1$.}  We partition $[n]\setminus \{x\}$ into three parts $V_1$, $V_2$ and $V_3$ with sizes $\vert V_1\vert= \vert V_2\vert=\vert V_3\vert=m$. A collection $\mathcal{E}$ of pairs of vertices from different parts of $[n]\setminus\{x\}$ is called \emph{admissible} if every vertex is contained in at most one pair from $\mathcal{E}$. Now let $F_1(\mathcal{E}, 3m+1)$ be the $3$-graph obtained from $F_1$ by deleting all triples $xuv$ and adding all tripartite triples $uvw$ for all $uv \in \mathcal{E}$. It is easy ti see that $F_1(\mathcal{E}, 3m+1)$  contains no $K_4$ covering $x$ and $\delta_2(F_1(\mathcal{E}, 3m+1))=\delta_2(F_1(3m+1))=2m-1$.

\noindent\textbf{Case 3: $n=3m+2$.}  We partition $[n]\setminus \{x\}$ into three parts $V_1$, $V_2$ and $V_3$ with sizes $\vert V_1\vert= \vert V_2\vert=m$ and $\vert V_3\vert=m+1$. A collection $\mathcal{E}$ of pairs of vertices from different parts of $[n]\setminus\{x\}$ is called \emph{admissible} if (i) every vertex $v\in V_1\sqcup V_2$ is contained in at most $2$ pairs from $\mathcal{E}$ and (ii) every vertex $v_3\in V_3$ is contained in at most $1$ pair from $\mathcal{E}$. Now let $F_1(\mathcal{E}, 3m+2)$ be the $3$-graph obtained from $F_1$ by deleting all triples $xuv$ and adding all tripartite triples $uvw$ for all $uv \in \mathcal{E}$. It is easy to see that $F_1(\mathcal{E}, 3m+2)$ contains no $K_4$ covering $x$ and $\delta_2(F_1(\mathcal{E}, 3m+2))=\delta_2(F_1(3m+2))=2m-1$.

There is yet another extremal construction. Partition $[n]\setminus \{x\}$ into three parts $V_1$, $V_2$ and $V_3$ with sizes $\vert V_1\vert=m-1$ and  $\vert V_2\vert=\vert V_3\vert=m+1$. In this context, a collection $\mathcal{E}$ of pairs of vertices from different parts of $[n]\setminus\{x\}$ is called \emph{admissible} if (i) every vertex $v_1\in V_1$ is contained in at most $3$ pairs from $\mathcal{E}$ and (ii) every vertex $v\in V_2\sqcup V_3$ is contained in at most $1$ pair from $\mathcal{E}$. Let $F'_1$ be the $3$-graph on $[n]$ consisting of all triples $xuv$, where $u, v$ come from different parts, and all triples of $[n]\setminus \{x\}$ that are not tripartite. Now let $F'_1(\mathcal{E}, 3m+2)$ be the $3$-graph obtained from $F'_1$ by deleting all triples $xuv$ and adding all tripartite triples $uvw$ for all $uv \in \mathcal{E}$. It is easy to see that $F'_1(\mathcal{E}, 3m+2)$ contains no $K_4$ covering $x$ and $\delta_2(F'_1(\mathcal{E}, 3m+2))=\delta_2(F_1(3m+2))=2m-1$.

We can show that the above constructions are \emph{all} extremal configurations for $n$ sufficiently large ($n\geq 999$). This can be done by first proving the following stability theorem. 
\begin{theorem}[Stability]\label{theorem: K4 stability}
	Suppose $n\geq 4$ and $0<\delta \leq\frac{1}{429}$. Suppose that $G$ is a $3$-graph on $n$ vertices with minimum codegree $\delta_2(G)\geq \left(\frac{2}{3}-\delta\right)n$ and that there is a vertex $x\in V(G)$ not contained in any copy of $K_4$ in $G$. Then there exists a tripartition $V_1\sqcup V_2\sqcup V_3$ of $V(G)\setminus \{x\}$ such that the following holds for all $i\in[3]$ and $j\ne i$:
	\begin{enumerate}[(i)]
		\item there is no triple in $G$ of the form $xV_iV_i$;
		\item all but at most $9\delta n^2$ triples of the form $xV_iV_{j}$ are in $G$;
		\item there are at most $4\delta n^3$ triples in $G$ of the form $V_1V_2V_3$;
		\item all but at most $6\delta n^3$ triples of the form $V_iV_iV_{j}$ are in $G$;
		\item $\left\vert \vert V_i\vert -\frac{n-1}{3}\right\vert\leq 2 \delta n$.
	\end{enumerate}
\end{theorem}
\begin{theorem}\label{theorem: extremal configurations}
\begin{itemize}
\item For  $n\equiv 0 \mod 3$ with $n\geq 858$, the extremal configurations for $c_2(n, K_4)$ are isomorphic to a subgraph of $F_1(\mathcal{E}, n)$ for some admissible $\mathcal{E}$.
\item For $n \equiv 1 \mod 3$ with $n \geq 715$, the extremal configurations for $c_2(n, K_4)$ are isomorphic to a subgraph of $F_1(\mathcal{E}, n)$ for some admissible $\mathcal{E}$.
\item For $n \equiv 2 \mod 3$ with $n \geq 1001$, the extremal configurations for $c_2(n, K_4)$ are isomorphic to a subgraph of $F_1(\mathcal{E}, n)$ or to a subgraph of $F_1'(\mathcal{E}, n)$ for some admissible $\mathcal{E}$.
\end{itemize}
\end{theorem}
The proof of Theorem~\ref{theorem: K4 stability} is very similar to that of the case $n=3m+2$ of Theorem~\ref{theorem: K4 threshold}, while the proof of Theorem~\ref{theorem: extremal configurations} is a straightforward application of parts (i) and (ii) of Theorem~\ref{theorem: K4 stability}.  We therefore defer these proofs to the appendix. 
%omit proofs here, and refer an interested reader to the appendix in the ArXiv version of this paper~\cite{FalgasRavryZhao15+} for details.

\section{Covering thresholds for other $3$-graphs}\label{section: other 3-graphs}
\subsection{$K_4^-$}
\begin{proof}[Proof of the lower bound in Theorem~\ref{theorem: k4-}]
We construct a $3$-graph $F_2(n)$ on $V=[n]$. Select a special vertex $x$. Split the remainder of the vertices into six parts $\sqcup_{i=1}^6V_i=V\setminus\{x\}$ with sizes as equal as possible, as follows:
\[\vert V_1\vert -1 \leq \vert V_6\vert \leq \vert V_5\vert \leq \vert V_4\vert \leq \vert V_3\vert \leq \vert V_2\vert \leq \vert V_1\vert.\]

Put as the link of $x$ the blow-up of a $6$-cycle through the six parts, i.e. add all triples of the form $xV_iV_{i+1}$ for $i\in[6]$, winding round modulo $6$ as necessary (identifying $V_7$ with $V_1$, and so on). Finally add those triples not involving $x$ which are not of type $V_iV_iV_{i+1}$, $V_iV_{i+1}V_{i+1}$ or $V_iV_{i+1}V_{i+2}$ for $i\in [6]$ (winding round modulo $6$) to form the $3$-graph $F_2(n)$.

%Finally, add in all triples of the following types: (a) $V_iV_iV_i$, (b) $V_iV_{i}V_{i+2}$, (c) $V_iV_{i}V_{i-2}$, (d) $V_iV_{i}V_{i+3}$, (e) $V_iV_{i+1}V_{i+3}$, (f) $V_iV_{i+1}V_{i+3}$, (g) $V_iV_{i+1}V_{i-2}$, (h) $V_iV_{i+2}V_{i-2}$.

%\textbf{[Add picture?]}
% I cannot think of a good way of representing the 3-graph in question, so maybe we should not have a piicture in this case?

Observe that the link graph of $x$ in $F_2(n)$ is triangle-free (being the blow-up of a $6$-cycle). Thus a putative $K_4^-$ containing $x$ would have to be induced by a $4$-set $\{a,b,c,x\}$, with $abc$, $abx$ and $acx$ all being triples of $F_2(n)$. Since $ab$ is in the link graph of $x$, we must have that $a,b$ come from different but adjacent parts $V_i, V_{i+1}$; by symmetry of $F_2(n)$, we may assume without loss of generality that $a\in V_1$ and $b\in V_2$. Since $acx \in E(F_2(n))$, it follows that $c\in V_2$ or $c\in V_6$. But by construction of $F_2(n)$, there are no triples of type $V_6V_1V_2$ or $V_1V_2V_2$, so that we cannot have in $abc \in E(F_2(n))$. Thus there is no copy of $K_4^-$ in $F_2(n)$ covering $x$.

Let us now compute the minimum codegree of $F_2(n)$. Consider vertices $a_i, a_i' \in V_i$, $a_{i+1}\in V_{i+1}$, $a_{i+2} \in V_{i+2}$ and $a_{i+3}\in V_{i+3}$. We have that $d(a_i, a_{i}')= n-3-\vert V_{i-1}\vert -\vert V_{i+1}\vert$, $d(a_i, a_{i+2})=n-3 -\vert V_{i+1}\vert$, $d(a_i, a_{i+3})=n-3$, and, lastly, 
\begin{align*}
d(a_i, x)=\vert V_{i-1}\vert + \vert V_{i+1}\vert \qquad\textrm{and} \qquad d(a_i, a_{i+1})= 1 + \vert V_{i+3}\vert +\vert V_{i+4}\vert.
\end{align*}
Up to the choice of $i$, this covers all possible pairs in $F_2(n)$. The first three quantities are at least $n-3-2\lceil\frac{n-1}{6}\rceil\geq \frac{2n}{3} -\frac{13}{3}$, which for $n\geq 12$ is greater than $\lfloor \frac{n-1}{3} \rfloor$. The last two quantities are both of order $\frac{n}{3}+O(1)$, however, and we analyse them more closely. Set $n=6m + r$ for some $r \in \{0,1,2,3,4,5\}$. Then
\begin{align*}
d(a_i,x)& \geq \min_i \bigl(\vert V_{i-1}\vert + \vert V_{i+1}\vert\bigr)= \vert V_6\vert +\vert V_4\vert=\left\{\begin{array}{ll}
2m-1 & \textrm{if $r=0$}\\
2m & \textrm{if $0<r <5$}\\
2m+1 & \textrm{if $r=5$,}
 \end{array}\right.
\end{align*}
and
\begin{align*} 
d(a_i, a_{i+1})& \geq \min_i \bigl(1+\vert V_{i+3}\vert +\vert V_{i+4}\vert\bigr)= 1+ \vert V_5\vert + \vert V_6\vert=\left\{\begin{array}{ll}
2m & \textrm{if $r=0$}\\
2m+1 & \textrm{if $0<r \leq 5$.}
 \end{array}\right.
\end{align*}
Thus
\begin{align*}
c_2(n, K_4^-)&\geq \delta_2(F_2(n))=\left\{\begin{array}{ll}
2m-1 & \textrm{if }r=0\\
2m & \textrm{if }0<r <5\\
2m+1 & \textrm{if }r=5.
\end{array} \right. %\\
%&=\left\{\begin{array}{ll} 
%\lfloor\frac {2n-3}{6}\rfloor & \textrm{if }n\not\equiv 4 \mod{6}\\
%\lfloor\frac {n-2}{3}\rfloor & \textrm{if }n\equiv 4 \mod{6}\\
\end{align*}
\end{proof}

\begin{proof}[Proof of the upper bound in Theorem~\ref{theorem: k4-}]
Let $G$ be a $3$-graph on $n\geq 4$ vertices. Suppose $\delta_2(G)> \frac{n}{3}$. Pick an arbitrary vertex $x\in V(G)$. Let $abx$ be any $3$-edge containing $x$. We have $d(a,b)+d(a,x)+d(b,x)-3>n-3$. So by the pigeonhole principle, there exists $c \in V(G)\setminus \{a,b,x\}$ which makes a $3$-edge of $G$ with at least two of $ab$, $ax$, $bx$. The $4$-set $abcx$ then contains a copy of $K_4^-$ in $G$ covering $x$, as required. This shows that $c_2(n, K_4^-) \leq \lfloor \frac{n}{3} \rfloor$. 
\end{proof}
\begin{remark} Again we actually proved something stronger here: our argument establishes that for $\delta_2(G)$ above $\lfloor \frac{n}{3} \rfloor$, every \emph{triple} of $E(G)$ can be extended to a copy of $K_4^-$.
\end{remark}
Matching the upper and lower bounds obtained above, we obtain the set of possible values for $c_2(n, K_4^-)$ claimed in Theorem~\ref{theorem: k4-}.\qed
% we see that, for $n=6m + \epsilon$, $\epsilon \in\{0,1,2,3,4,5\}$,
%\begin{align*}
%c_2(n, K_4^-)&=\left\{\begin{array}{ll}
%2m-1 \textrm{ or } 2m& \textrm{if }\epsilon=0\\
%2m & \textrm{if }\epsilon \in \{1,2\}\\
%2m \textrm{ or } 2m+1& \textrm{if }\epsilon\in \{3,4\}\\
%2m+1& \textrm{if }\epsilon=5.
%\end{array}\right.
%\end{align*}

% % % %%\textbf{Worth thinking whether some clever, short, non-stability argument could deal with one of the remaining open cases. }
\begin{remark}
	We believe the gap between the upper and lower bounds for $c_2(n, K_4^-)$ could be closed using similar (but more involved) stability arguments to those we used on to determine $c_2(n, K_4)$. However since such arguments would be non-trivial (the conjectured extremal configurations in this case are $6$-partite) and would greatly increase the length of this paper, we do not pursue them here and leave open the determination of $c_2(n, K_4^-)$ in the case where $n\equiv 0,3,4 \bmod{6}$.
\end{remark}
\subsection{$C_5$}
\begin{proof}[Proof of the lower bound in Theorem~\ref{theorem:  c5}]
	We construct a $3$-graph $F_3(n)$ on $V=[n]$. Select a special vertex $x$. Split the remainder of the vertices into two parts $V\setminus\{x\}=V_1\sqcup V_2$ with sizes as equal as possible, $\vert V_2\vert-1 \leq \vert V_1\vert \leq \vert V_2 \vert$. Form the link graph of $x$ by adding in all pairs internal to one of the parts, i.e.  all pairs of the form $xV_1V_1$ or $xV_2V_2$. Next, add in all triples not containing $x$ and meeting both of the parts, i.e. all pairs of the form $V_1V_1V_2$ or $V_1V_2V_2$. This yields a $3$-graph $F_3(n)$ with minimum codegree
	%\begin{align*}
	$\delta_2(F_3(n))=\vert V_1\vert -1 = \lfloor \frac{n-3}{2}\rfloor$,  attained by $x$ and any vertex $a\in V_1$; see Figure~2.

	Now there is no copy of $C_5$ covering $x\in F_3(n)$. Indeed, let $S=\{a_1,a_2,b_1,b_2\}$ be a set of four distinct vertices in $V\setminus \{x\}$ such that all of $a_1a_2x$, $a_1b_1x$ and $b_1b_2x$ are triples of $F_3(n)$. Then by construction these four vertices must all lie within the same part of $F_3(n)$. But by construction again we have that $S$ spans no triple of $F_3(n)$, whence $S\cup\{x\}$ does not contain a copy of $C_5$. 
\end{proof}
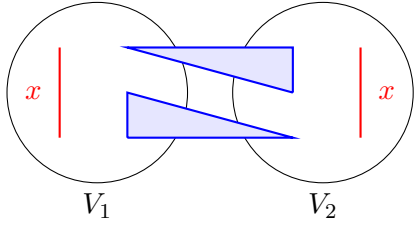
\begin{figure}\label{figure: c5 construction}

\centering
\begin{tikzpicture}

\draw (0,0) circle [radius=1.2];
\draw (3,0) circle [radius=1.2];

%\coordinate (delta) at ($(1.5, 5)$);
%\fill[red] (delta) circle (2pt) node[above] {$x$};

\coordinate (a1) at (0.4, -0.6);
\coordinate (a2) at (0.4, 0);
\coordinate (a3) at (0.4, 0.6);

\coordinate (b1) at (2.6, -0.6);
\coordinate (b2) at (2.6, 0);
\coordinate (b3) at (2.6, 0.6);

\draw[blue!100, thick, fill=blue!10] (a1)-- (a2) -- (b1) -- (a1);
\draw[blue!100, thick, fill=blue!10] (b2)-- (b3) -- (a3) -- (b2);

\coordinate (A1) at (-0.5, 0.6);
\coordinate (A2) at (-0.5, -0.6);
\draw[thick, red!100] (A1) -- (A2); 

\coordinate (B1) at (3.5, 0.6);
\coordinate (B2) at (3.5, -0.6);
\draw[thick, red!100] (B1) -- (B2); 

\node (aedge) at (-0.85,0)[red] {$x$};
\node (bedge) at (3.85 , 0)[red] {$x$};

\node (an) at (0, -1.5)[] {$V_1$};
\node (bn) at (3, -1.5)[] {$V_2$};
\end{tikzpicture}
\caption{The $3$-graph $F_3(n)$. The red pairs and the blue triples make up the link graph of $x$ and the remainder of $E(F_3)$ respectively.}
\end{figure}

\begin{proof}[Proof of the upper bound in Theorem~\ref{theorem:  c5}]
	Let $G=(V,E)$ be a $3$-graph on $n$ vertices with minimum codegree $\delta_2(G)>\frac{n}{2}$. Let $x$ be any vertex. Fix an edge $ab$ in the link graph $G_x$. Since $\delta_1(G_x)> n/2$, $a$ and $b$ each has at least $\frac{n}{2} -1$ neighbours in $V\setminus \{x, u, v\}$. Hence $a$ and $b$ have a common neighbour $c$ in $G_x$.
	% Let $ be a triple of vertices from $V\setminus \{x\}$ inducing a triangle in $G_x$. 
	We shall use the triangle $\{a,b,c\}$ to find a copy of $C_5$ covering $x$. For this purpose, it is convenient to introduce the following notation. Given a $4$-set of vertices $\{y_1,y_2,z_1,z_2\}$ from $V\setminus \{x\}$, write $y_1y_1\vert z_1z_2$ as a shorthand for the statement that all of $y_1z_1x$, $y_1y_2x$, $y_2z_2x$, $y_2z_1z_2$ and $y_1z_1z_2$ are in $E(G)$ (and in particular that $\{x,y_1,z_1,z_2,y_2\}$ contains a copy of $C_5$ covering $x$).

	\begin{lemma}
		There is either a copy of $C_5$ or a copy of $K_4$ covering $x$ in $G$.
	\end{lemma}
	\begin{proof}
		If  $abc\in E(G)$ then the claim is immediate since $S=\{a,b,c,x\}$ induces a complete $3$-graph. Assume therefore that $abc\notin E(G)$. By our codegree assumption,
		\begin{align*}
		d(a,b)+d(a,c)+d(b,c)+d(a,x)+d(b,x)+d(c,x) -9&\geq 6\delta_2(G)-9>3(n-4). 
		\end{align*}
		Thus there exists $y\in V\setminus S$ which makes a $3$-edge with at least four of the pairs $ab$, $ac$, $bc$, $ax$, $bx$, $cx$. It is now easy to check that $S\cup \{y\}$ contains either a $K_4$ or a $C_5$ covering $x$. Indeed by symmetry we may reduce the case-checking to the following three possibilities:
		\begin{itemize}
			\item if $y$ makes a $3$-edge with $ab$, $bc$, $ac$ and $ax$, then $ab\vert cy$;
			\item if $y$ makes a $3$-edge with $ab$, $ac$ and at least one of $bx$ or $cx$, then $bc\vert ay$;
			\item if $y$ makes a $3$-edge with $ab$ and with both of $ax$ and $bx$, then $\{a,b,x,y\}$ induces a copy of $K_4$. \qedhere
		\end{itemize} 
	\end{proof}

	With a view towards proving Theorem~\ref{theorem:  c5}, we may thus assume that there is a copy of $K_4$ covering $x$. Let $S=\{a,b,c,x\}$ be a $4$-set of vertices inducing such a $K_4$. By the codegree assumption,
	\begin{align*}
	d(a,b)+d(a,c)+d(b,c)+d(a,x)+d(b,x)+d(c,x) -12&\geq 6\delta_2(G)-12>3(n-4). 
	\end{align*}
	Thus there exists $y\in V\setminus S$ which makes a $3$-edge with at least four of the pairs $ab$, $ac$, $bc$, $ax$, $bx$, $cx$. It is now easy to check that $S\cup \{y\}$ contains a copy of $C_5$ covering $x$. Indeed by symmetry we may reduce the case-checking to the following three possibilities:
	\begin{itemize}
		\item if $y$ makes a $3$-edge with $ab$, $ac$, $bc$ and $ax$, then $ab\vert cy$;
		\item if $y$ makes a $3$-edge with $ab$, $ac$ and at least one of $bx$ and $cx$, then $bc\vert ay$;
		\item if $y$ makes a $3$-edge with $ab$ and with all of $ax$, $bx$ and $cx$, then  $cy\vert ab$.
	\end{itemize}
	In all three cases we cover $x$ with a copy of $C_5$. The claimed upper bound on $c_2(n, C_5)$ follows.
\end{proof}

\begin{remark}\label{remark: c5 unstable}
Interestingly, as pointed out to us by Jie Han and Allan Lo, another very different construction attains the lower bound in Theorem~\ref{theorem:  c5}. Take a balanced bipartition of $[n]$ into two sets $V_1$ and $V_2$, with $\vert V_1\vert  \leq \vert V_2 \vert$. Now take all triples meeting $V_1$ in an even number of vertices to form a $3$-graph $F_4(n)$. Note that $\delta_2(F_4(n))=\min\left(\vert V_1 \vert -1, \vert V_2\vert -2 \right)$ (attained by pairs from $A\times B$ and $B^{(2)}$ respectively), which is exactly equal to $\lfloor\frac{n-3}{2}\rfloor$ . Now, it is an easy exercise to check that every vertex $x\in V_1$ fails to be covered by a $C_5$, giving us a second proof that $c_2(n, C_5)\geq \lfloor \frac{n-3}{2}\rfloor$. In particular, we do not have stability for this problem: we have two near-extremal constructions which are easily seen to lie at edit distance $\Omega(n^3)$ from each other. Also we have that just below the codegree threshold for covering by $C_5$, we could have as many as $\lfloor \frac{n}{2}\rfloor$ uncovered vertices. This stands in sharp contrast with the situation for $K_4$ (see the discussion in Section~\ref{section: conclusion}).
\end{remark}
%OLD ARGUMENT: 2/3 BOUND FOR COVERING EVERY PAIR WITH A C5
%Let $G$ be a $3$-graph on $n\geq 5$ vertices with $\delta_2(G)\geq\frac{2n-1}{3}$. Let $x, a_1$ be arbitrary vertices. Let $a_2, b_1$ be distinct vertices in $\Gamma(x,a_1)$. Since
%\[d(x,b_1)+d(a_2, b_1)+d(a_1,a_2)\geq 2n-1= 2(n-4)+ 3\cdot 2+1, \]
%there exist $b_2 \in V\setminus \{x,a_1,a_2,b_1\}$ such that all of $xb_1b_2$, $b_1b_2a_2$ and $a_1a_2b_2$ are $3$-edges of $G$. The subgraph of $G$ induced by the $5$-set $\{x,a_1,a_2,b_1,b_2\}$ thus contains a copy of $C_5$ covering $x$. This establishes that $c_2(n, C_5)<\frac{2n-1}{3}$.
%\end{proof}
%\begin{remark}
%The argument above in fact establishes that if $\delta_2(G)\geq \frac{2n-1}{3}$ then every \emph{pair} in $G$ can be covered by a copy of $C_5$ 
%\end{remark}

\subsection{$K_5^-$}
\begin{proof}[Proof of Theorem~\ref{theorem: k5-}]
For the lower bound, note that
\[c_2(n, K_5^-)\geq c_2(n, K_4)\geq \delta_2(F_1(n))=\left\lfloor \frac{2n-5}{3}\right\rfloor.\]
For the upper bound, let $G$ be a $3$-graph on $n$ vertices with $\delta_2(G)> \frac{2n-2}{3}$. By Theorem~\ref{theorem: K4 threshold}, for any vertex $x\in V(G)$ there is a triple $a_1,a_2,a_3$ such that $S=\{x,a_1,a_2,a_3\}$ induces a copy of $K_4$ in $G$. Now
\[d(x,a_1)+d(x,a_2)+d(x,a_3)+d(a_1,a_2)+d(a_1,a_3)+d(a_2,a_3)-12> 4(n-4),\]
whence there exist $a_4 \in V\setminus S$ which makes a $3$-edge with at least $5$ of the pairs from $S^{(2)}$. Thus $S\cup\{ a_4\}$ contains a copy of $K_5^-$ covering $x$. This shows that $c_2(n, K_5^-)\leq \lfloor\frac{2n-2}{3}\rfloor$.
\end{proof}

%\textbf{Worth checking whether with a bit more work we could determine the exact threshold.}

\subsection{The Fano plane}
\begin{proposition}\label{prop: fano bounds}
$\lfloor \frac{n}{2}\rfloor\leq c_2(n, \mathrm{Fano})\leq \lfloor\frac{2n}{3}\rfloor$.
\end{proposition}

\begin{proof}
The lower bound is from the codegree threshold of the Fano plane: consider a bipartition of $[n]$ into two sets $V_1\sqcup V_2$ with $\vert V_1\vert =\lfloor \frac{n}{2}\rfloor$ and $\vert V_2\vert = \lceil \frac{n}{2}\rceil$, and adding all triples meeting both parts. The resulting $3$-graph is easily seen to be $\mathrm{Fano}$-free (it is $2$-colourable, whereas the Fano plane is not) and has codegree $\lfloor \frac{n}{2}\rfloor$. For the upper bound, apply Lemma~\ref{lemma: trivial upper bound} with $F=\mathrm{Fano}$ and $r=3$.
%For the upper bound, suppose $G$ is an $n$-vertex $3$-graph with minimum codegree $\delta_2(G) > \frac{2n}{3}$, for some $n\geq 7$. Let $x\in V(G)$ be an arbitrary vertex of $V(G)$.
%Let $a_1a_2$ and $b_1b_2$ be disjoint edges in the link graph of $x$ in $G$. Since \[d(a_1,b_1)+d(a_2,b_2)\geq \frac{4n}{3}>n+1,\]
% there exists a vertex $c_1 \in V(G)\setminus\{a_1,a_2,b_1,b_2,x\}$  for which $a_1b_1c_1$ and $a_2b_2c_1$ are both in $E(G)$.
%Now each pair of vertices in  $S=\{a_1,a_2,b_1,b_2,c, x\}$ makes a $3$-edge with at most $4$ other vertices from $S$. Since
%\[d(x,c_1)+ d(a_1,b_2)+d(a_2,b_1)-3\cdot 4> 2(n-6), \]
%there exists $c_2\in V\setminus S$ such that all of $xc_1c_2$, $a_1b_2c_2$, $a_2b_1c_2$ are in  $E(G)$. In particular the subgraph of $G$ induced by $\{a_1,a_2,b_1,b_2,c_1,c_2,x\}$ contains a copy of the Fano plane as a subgraph, and this copy covers $x$. This shows that $c_2(n, \mathrm{Fano})\leq\lfloor\frac{2n}{3}\rfloor$.
\end{proof}
%\begin{remark}
%In fact we showed that if $\delta(G)>\frac{2n}{3}$  then every pair of $3$-edges of $G$ meeting at a single vertex can be extended to a Fano plane.
%\end{remark}

\subsection{$F_{3,2}$}
\begin{theorem}\label{prop: f32 bounds}
${1}/{3}\leq c_2(F_{3,2})\leq  3/7$.
\end{theorem}
\begin{proof}
The lower bound is from the codegree density of $F_{3,2}$. An $F_{3,2}$-free construction on $n$ vertices with codegree $\lfloor \frac{n}{3}\rfloor-1$ is obtained by considering a tripartition of $[n]$ into three parts with sizes as equal as possible, $\vert V_3\vert-1\leq \vert V_1\vert \leq \vert V_2\vert \leq \vert V_3\vert$ and adding all triples of the form $V_iV_iV_{i+1}$ (this is not actually best possible --- see~\cite{FalgasRavryMarchantPikhurkoVaughan15} for a determination of the precise codegree threshold and the extremal constructions attaining it).

For the upper bound, let $G$ be a $3$-graph on $n$ vertices with $\delta_2(G)= cn$. Suppose there exists $x\in V(G)$ such that there is no copy of $F_{3,2}$ in $G$ covering $x$. This means that for every vertex $v\in V\setminus \{x\}$,  $\Gamma(x,v)$ is an independent ($3$-edge--free) set in $G$, and moreover that for every $4$-set $\{a,b,c,d\}\subseteq V(G)$, at least one of the triples $\{xab, xcd, abc,abd\}$ is not in $E(G)$. For convenience, we shall write $ab\vert cde$ as a short-hand for the statement that $\{abc, abd, abe, cde\}$ all are $3$-edges of $G$.

We use the following technical lemma to deduce $c\le 3/7+o(1)$. 
\begin{lemma}\label{lem:AB}
If there exist sets $A, B\subseteq V$ such that
\begin{enumerate}
\item $A$ is a subset of $\Gamma(x,y)$ of size $cn$ for some $y\in V\setminus\{x\}$, and $B$ is a subset of $V\setminus (A+\cup\{x\})$ of size $cn$, and
\item $B$ is independent in $G$ and the link graph $G_x$,
\end{enumerate}
then $c\le 3/7+o(1)$.
\end{lemma}

\begin{proof}[Proof of Lemma~\ref{lem:AB}]
Let $C= V\setminus \left(A\cup B\right)$. We have $\vert C\vert = n(1-2c)$. 
By our assumption, $A$ is independent in $G$. 
By the codegree assumption, at least $\binom{\vert A\vert}{2}\left(cn -\vert C\vert\right)$ triples of $G$ have two vertices in $A$ and one vertex in $B$. Consequently, at most $\binom{\vert A\vert}{2} \vert C\vert = \binom{cn}{2} (1-2c)n$ triples of the form $AAB$ are missing from $G$.

On the other hand, let $b, b'\in B$. Since $B$ is independent in $G$ and $G_x$, we have $\Gamma(b, b')\subseteq A\cup C$ and $\Gamma(b, x)\subseteq A\cup C$. Consequently $\vert \Gamma(b',x)\cap A\vert \geq cn -\vert C\vert \geq (3c-1)n$ and 
\begin{align*}
\vert \Gamma(b,b')\cap \Gamma(b,x)\cap A\vert &\geq 2(cn - \vert C\vert) -\vert A\vert\geq (5c-2)n.
\end{align*}
For any $a\in \Gamma(b,b')\cap \Gamma(b,x)\cap A$ and any $a' \in \Gamma(b',x)\cap A$, the triple $aa'b$ must be absent from $G$ -- otherwise $ab\vert a'bx$. There are at least $\binom{ (3c-1)n }2 - \binom{ (1-2c)n }2$ such pairs $(a,a')$ because, in general, there are at least $\binom{|A_1|}2 - \binom{|A_1| - |A_2|}2$ pairs $(a_1, a_2)$ with $a\in A_1$ and $a_2\in A_2$ for arbitrary sets $A_1, A_2$ satisfying $|A_1|\ge |A_2|$.

There are thus at least $\binom{ (3c-1)n }2 - \binom{ (1-2c)n }2$ distinct pairs $(a,a')$ for which $aa'b\notin E(G)$. Summing over all $b\in B$, this gives us a total of at least $\left( \binom{ (3c-1)n }2 - \binom{ (1-2c)n }2 \right)cn$ $AAB$ triples missing from $E(G)$. Combining this together with our upper bound on the number of missing $AAB$ triples yields the inequality
\[
\left(\binom{ (3c-1)n }2 - \binom{ (1-2c)n }2 \right)cn \le \binom{cn}2 (n(1-2c)-1),
\]
which implies that
\[
\left( (3c-1)^2 - (1- 2c)^2 \right) \frac{n^2}2 cn \le \frac{ c^2 n^2} 2 \left(1 - 2c \right) n+O(n^2).
\]
 This inequality in turn gives $c\le 3/7+o(n^{-1})$.
\end{proof}

%\begin{claim}\label{clm:Gx}
%The link graph $G_x$ is triangle-free.
%\end{claim}
We now show that we can find $A,B \subseteq V$ satisfying the properties in Lemma~\ref{lem:AB}.

Suppose first of all that $G_x$ is not triangle-free. Let $y a_1 a_2$ vertices spanning a triangle in $G_x$. 
Let $A$ be a subset of $\Gamma(x,y)$ of size $cn$. Then $A$ must be an independent set in $G$.
Let $B$ be a subset of $\Gamma(a_1,a_2)$ in $V\setminus \{x\}$ of size $cn$. Then $B$ is disjoint from $A$ and is an independent set in $G_x$ --- indeed if $b_1b_2 \in G_x$ for some $b_1,b_2\in B$ then $a_1a_2\vert xb_1b_2$, a contradiction. We now show that $B$ is an independent set in $G$. Indeed, for every $b\in B$, $\Gamma(b,x)$ is a subset of $V\setminus B$ of size at least $cn$.  Consider an arbitrary triple $\{b_1,b_2,b_3\}$ of distinct vertices from $B$. Since
\begin{align*}
d(b_1,x)+d(b_2,x)+d(b_3, x) -2 \left(n-\vert B\vert\right)&\geq \left(3c -2(1-c)\right)n=(5c-2)n>0,
\end{align*}
by the pigeon-hole principle there exists $a\in A\cup C$ with $xab_1, xab_2, xab_3$ all in $E(G)$. In particular $b_1b_2b_3 \notin E(G)$, as otherwise we would have $ax\vert b_1b_2b_3$. It follows that $B$ must be an independent set in $G$. Thus $A, B$ satisfy the two properties in Lemma~\ref{lem:AB}, and thus $c\leq 3/7 +o(1)$

On the other hand, suppose $G_x$ was triangle-free. Let $y\in V\setminus \{x\}$, and let $A$ be a subset of $\Gamma(x,y)$ of size $cn$. Since $G_x$ is triangle-free and $x$ is not covered by an $F_{3,2}$-subgraph, 
$A$ forms an independent set in both $G_x$ and $G$. Let $a\in A$ be arbitrary, and let $B$ be a subset of $\Gamma(a,x)$ of size $cn$. Then $B$ is disjoint from $A$ and independent in $G_x$ (since $G_x$ is triangle-free). Thus $A, B$ satisfy the two properties in Lemma~\ref{lem:AB}, and $c\leq 3/7+o(1)$. Thus $c_2(F_{3,2})\leq 3/7$ as claimed.
\end{proof}

%\textbf{We still have no better lower bound than $\gamma(F_{3,2})=\frac{1}{3}$ for $c_2(F_{3,2})$. Moreover our upper bound above is not tight. For $c<\frac{2}{5}$, we no longer can guarantee that the set $B$ is independent in $G$ in Claim~\ref{claim: Gx triangle free} but it should still remain `mostly' independent in $G$, and the counting argument, which has quite a lot of slack, should still carry through. Moreover the argument just ignores the left-over set $C$; it cannot possibly take up all the missing edges without generating some $x$-covering copies of $F_{3,2}$ on its own.}

\subsection{$K_t$, $t\geq 5$}\label{subsection: kt bounds}

\begin{proposition}\label{prop:Kt}
For all $t\ge 5$, $c_2(K_t) \leq 1 - 1/\binom{t-1}{2}$.
\end{proposition}
\begin{proof}
Applying Lemma~\ref{lemma: trivial upper bound} with $F=K_t$ and $r=\binom{t-1}2$, we get 

$c_2(n, K_t)\leq \left\lfloor \left(1 - \frac{1}{\binom{t-1}2} \right)n - \frac{2t-6}{t-2} \right\rfloor$. \qedhere
\end{proof}

%suppose that $G$ is a $3$-graph on $n\geq 6$ vertices with $\delta_2(G)> \frac{9n-15}{10}$. Let $x$ be an arbitrary vertex in $G$. By Proposition~\ref{prop: k5 bounds} above, there is a $5$-set $S$ containing $x$ and inducing a copy of $K_5$ in $G$. Then 
%\[\sum_{uv\in S^{(2)}} (d(u,v)-3) > 9n-15-30=9(n-5),\]
%whence there exists $a\in V\setminus S$ which makes a $3$-edges with all $10$ pairs from $S^{(2)}$. The $6$-set $S\cup \{a\}$ then induces a copy of $K_6$ in $G$ covering $x$. Thus $c_2(n, K_6) \leq \lfloor \frac{9n-15}{10}\rfloor$, and $c_2(K_5) \leq \frac{9}{10}$ as claimed.

We now derive a lower bound for the covering codegree density of $K_{t}$ by using (small) lower-bound constructions for the codegree threshold of $K_{t-1}$ 

\begin{proposition}\label{prop: blow-up of a codegree construction}
Suppose there exists a $K_{t-1}$-free $3$-graph $H$ on $[m]$ with minimum codegree $\delta$. Then $c_2(K_{t})\geq (\delta + 2)/{m}$.
\end{proposition}
\begin{proof}
%Let $H$ be a $K_t$-free $3$-graph on $[m]$ with $\delta_2(H)\geq \delta m$. 
We build a $3$-graph $G$ on $n=Nm+1$ vertices as follows. Set $V=[n]$ and set aside a special vertex $x$. Partition $V\setminus \{x\}$ into $m$ sets $V_1, \dots, V_m$, each of size $N$. Set as the link graph of $x$ all pairs of vertices from distinct parts. For every triple $ijk \in E(H)$, add to $G$ all $3$-edges of the form $V_iV_jV_k$. Finally, add all triples of $V\setminus \{x\}$ that contain at least two vertices from one part.
The minimum codegree of $G$ is 
\[
\delta_2(G)= (\delta + 2) N - 1 \ge \frac{\delta + 2}{m} n -2. 
\] Now consider a $(t-1)$-set $S\subset V\setminus \{x\}$ that induce a $t$-clique in the link graph of $x$. By construction, these vertices must come from $t$ different parts of $V\setminus \{x\}$. Since $H$ is $K_{t-1}$-free, by our construction, some triple of $S$ is absent from $G$. Thus $S\cup\{x\}$ does not induce a copy of $K_{t}$ in $G$. Taking the limit as $n\rightarrow \infty$, the result follows.
\end{proof}

\begin{corollary}\label{prop: general bounds}
Let $t\ge 4$.
\begin{enumerate}
\item $c_2(K_t)\geq \frac{t-2}{t-1}$, and
\item $c_2(K_t)\geq \frac{2t-6}{2t-5}$ if $t\equiv 0, 1 \mod 3$.
\end{enumerate}
\end{corollary}

\begin{proof}
For Part~1, we apply Proposition~\ref{prop: blow-up of a codegree construction} with $H = K_{t-1}^-$ (thus $m= t-1$ and $\delta= t-4$) and obtain $c_2(K_t)\geq \frac{t-2}{t-1}$. 

For Part~2, since $t\equiv 0,1\mod 3$, we have $2t-5 \equiv 1,3 \mod 6$, whence there exists a Steiner triple system $\mathcal{S}$ on the vertex set $[2t-5]$. It is easy to see that every set $T \subset [2t-5]$ of $t-1$ vertices spans at least one triple from $\mathcal{S}$. Indeed, fix a vertex $a\in T$: all the pairs of $T$ containing $a$ must have distinct neighbours under $\mathcal{S}$ in $[2t-5]\setminus T$. Since $t-2> (2t-5)-(t-1)$, this is impossible. Therefore the complement $3$-graph $\overline{\mathcal{S}}$ is $K_{t-1}$-free. Applying Proposition~\ref{prop: blow-up of a codegree construction} with $H = \overline{\mathcal{S}}$ (thus $m= 2t-5$ and $\delta= 2t-8$), we obtain that $c_2(K_t)\geq \frac{2t-6}{2t-5}$.
\end{proof}

\begin{remark}
\begin{enumerate}
\item Combining Proposition~\ref{prop:Kt} and Corollary~\ref{prop: general bounds} gives that $\frac{3}{4} \leq c_2(K_5) \leq \frac{5}{6}$ and $\frac{6}{7} \leq c_2(K_6) \leq \frac{9}{10}$.
\item Theorem~\ref{theorem: K4 threshold} shows that the lower bound in Proposition~\ref{prop: general bounds} is tight in the case $t=4$. The bound is also tight in the trivial case $t=3$, since $c_2(n, K_3)=1=o(n)$. If this bound is tight in general, then we do not have stability for the covering codegree-threshold problem: while the $3$-edge $K_3$ and the Fano plane are the unique (up to isomorphism) Steiner triple systems on $3$ and $7$ vertices respectively, there are for example $11,084,874,829$ non-isomorphic Steiner triple systems on $19$ vertices (see~\cite[Section 4.5]{ColbournDinitz06}).
\end{enumerate}
\end{remark}

\section{Concluding remarks}\label{section: conclusion}
There are many questions arising from our work. To begin with, we may ask which of the fundamental properties of Tur\'an density and codegree density does the covering codegree density $c_2$ share. Explicitly:
\begin{enumerate}[1.]
\item do we have \emph{supersaturation}? That is, if $\delta_2(G)\geq c_2(n,F) +\varepsilon n$ for some fixed $\varepsilon>0$, is it the case that every vertex in $G$ is contained in $\Omega(n^{\vert V(F) \vert -1})$ copies of $F$?
\item do we have \emph{blow-up invariance}? Given a $3$-graph $F$, we define the blow-up $F(t)$ to be the $3$-graph on $V(F)\times [t]$ with $3$-edges $\{(u,i)(v,j)(w,k): \ uvw\in E(F), \ i,j,k\in[t]\}$. Is it the case that for every $F$ and every fixed $t$ we have $c_2(F)=c_2(F(t))$?
\item is the set of covering codegree densities $\{c_2(F): \ F \textrm{ a }3\textrm{-graph}\}$ dense in $[0,1]$, or does it have jumps? 
\end{enumerate}
The first two of these questions are addressed in a forthcoming work of the authors.
%\noindent\textbf{(Actually the three questions above seem quite suitable for a follow-up paper, so should we say `these are the subject of future work' or something like that?)}
In addition there are some natural variants of the covering codegree threshold $c_2(n,F)$ which may be interesting. What if instead of covering every vertex by a copy of $F$ we wanted to cover every pair? What if we wanted instead to be able to extend every $3$-edge to a copy of $F$? It is not immediately clear whether the corresponding codegree-extremal functions behave similarly to $c_2(n,F)$ or not.

In a different direction, what if we asked for the threshold for covering all but at most $k$ vertices, for some $k\geq 1$? On the one hand, in the case of $C_5$ we observed in Remark~\ref{remark: c5 unstable} that this does not affect the value of the covering threshold very much. On the other hand, Theorem~\ref{theorem: K4 stability} implies that the threshold for covering all but at most $1$ vertex with a copy of $K_4$ is at most $\left({2}/{3}-c\right)n$ for some $c>0$ (therefore the problem is genuinely different from $c_2(n, K_4)$). Let us sketch a proof. Let $G$ be a 3-graph with $\delta_2(G)\ge \left({2}/{3}-c\right)n$ for some $c>0$ sufficiently small. Suppose that $x\in V(G)$ is not covered by any copy of $K_4$. By Theorem~\ref{theorem: K4 stability}, there is a partition $V_1\sqcup V_2\sqcup V_3$ of $V(G)\setminus \{x\}$ satisfying (i)--(v). If another vertex $y$ is not covered by any copy of $K_4$, then there is a partition $V'_1\sqcup V'_2\sqcup V'_3$ of $V(G)\setminus \{y\}$ satisfying (i)-(v) as well. Because of (iii) and (iv), these two partitions essentially coincide. 
Now consider $\Gamma(x,y)$, which has size at least $(2/3-c) n$. There are about $(2/3 - c) (1/3 - 3c) n^2/2 $ pairs $u, v\in \Gamma(x, y)$ coming from different parts of $V_1\cap V'_1, V_2\cap V'_2, V_3\cap V'_3$. Since $(2/3 - c) (1/3 - 3c) n^2/2 > 2 \cdot 10c n^2$ (for $c$ sufficiently small), by (ii), there exists a pair $u,v\in \Gamma(x, y)$ such that both $uvx$ and $uvy$ are edges of $G$. This implies that $\{u,v,x,y\}$ spans a copy of $K_4$, a contradiction. The authors note that the bound on $c$ given by this argument can be significantly improved; this is the subject of future work.
%What is the correct threshold in that case, and would allowing more than one uncovered vertex lower it even further? 

Finally, it would be interesting to determine the value of $c_2(F)$ when $F$ is the Fano plane or $F_{3,2}$, and to have if not a tight result then at least a reasonable guess as to the value of $c_2(K_t)$ for $t\geq 5$. An investigation of $c_1(n, F)$ when $F=K_4^-$ and $F=K_4$ would also be desirable.

We should note here that for such small $3$-graphs $F$ the problem of proving upper bounds for $c_1$ or $c_2$ should be amenable to flag algebra computations by following the approach of~\cite{FalgasRavryMarchantPikhurkoVaughan15} to encode the minimum degree/codegree constraint. Note however that one will need to do computation with non-uniform hypergraphs, containing a mixture of $2$-edges (from the link graph of an uncovered vertex $x$) and $3$-edges.

\section{Acknowledgements}
The authors thank Jie Han and Allan Lo for pointing out the alternative lower bound construction for Theorem~\ref{theorem:  c5}. The first author is also grateful for the hospitality of Georgia State University during a research visit which led to the present paper.

\bibliographystyle{plain}
\bibliography{codegreecoveringbiblio}

\section{Appendix: proof of Theorem~\ref{theorem: K4 stability} and Theorem~\ref{theorem: extremal configurations}}
\begin{proof}[Proof of Theorem~\ref{theorem: K4 stability}]
We run through the proof of the case $n=3m+2$ of Theorem~\ref{theorem: K4 stability}, replacing the codegree assumption $\delta_2(G)=2m-2$ by the assumption $\delta_2(G)\geq \frac{2n}{3}-\delta n$. This gives us new versions of our claims and lemmas with error terms involving $\delta n$, and conditions on $n$ being sufficiently large replaced by conditions on $\delta$ being sufficiently small. For the sake of completeness, we derive them below.

Let $G$ be a $3$-graph on $n$ vertices with $\delta_2(G)\geq \left(\frac{2}{3}-\delta\right)n$, for some $\delta$: $0<\delta \leq 1/429$. Suppose there is a vertex $x$ of $G$ not contained in any copy of $K_4$. As $2/3-\delta > 1/2$, the link graph $G_x$ contains a triangle $\{ab, bc, ac\}$. Set $S=\{a,b,c,x\}$ and for each $y\in V\setminus S$, define $S_y$ as in Lemma~\ref{lem:Sy}. By Lemma~\ref{lem:Sy}, $S_y$ is a subset of $S^{1, c}, S^{1,b}$, $S^{1,a}$, $S^{2,a}$, $S^{2,b}$, $S^{2,c}$ or $S^3$. 
	For $i\in\{1,2\}$ and $j\in\{a,b,c\}$, write $s_{i,j}$ for the number of vertices $y \in V\setminus S$ for which $S_y = S^{i,j}$, and write $s_i$ for the sum $s_{i,a}+s_{i,b}+s_{i,c}$. Finally let $s_0$ be the number of vertices $y\in V\setminus S$ such that $S_y \neq S^{i,j}$ for any $i\in \{1,2\}$ and $j\in\{a,b,c\}$. Note that $|S_y|\le 3$ for such $y$. We know that $s_1+s_2+s_0 = n-4$. Furthermore, by the codegree assumption, 
	\begin{equation}\label{ax-bx-cx new}
	3\, \left(\frac{2n}{3} -\delta n\right)  \leq d(a,x)+ d(b,x)+ d(c,x)\leq 2s_1+s_2+3s_0 +6,
	\end{equation}
	\begin{equation}\label{ax-bx-cx-ab-bc-ac new}
	6\, \left(\frac{2n}{3} -\delta n\right) \leq d(a,x)+ d(b,x)+ d(c,x)+d(a,b)+d(b,c)+d(c,a) \leq 4s_1+4s_2 +3s_0 +9,
	\end{equation}
	Substituting $s_0=n-4-s_1-s_2$ into (\ref{ax-bx-cx new}) and  (\ref{ax-bx-cx-ab-bc-ac new}) yields that $s_1+2s_2 \le n+3\delta n -6$ and $s_1+s_2  \geq n-6\delta n +3$, respectively. Combining the two inequalities we have just obtained, we get
	\begin{align*}
	s_2 \leq 9\delta n -9 \quad \text{and} \quad s_1 \geq n- 15\delta n +12.
	\end{align*}
	We now show as before that the weight of $s_1$ splits almost equally between $s_{1,a}$, $s_{1,b}$, $s_{1,c}$. Note that
	\begin{align*}
	\frac{2n}{3}-\delta n \leq d(b,c)\leq n-3-s_{1,a}, % YZ changed -2 to -3 because $abc\not\in E(G)$
	\end{align*}
	from which it follows that $s_{1,a}\leq \frac{n}{3}+\delta n -3$. Similarly we derive that $s_{1,b,} s_{1,c}\le \frac{n}{3}+\delta n -3$. Consequently
	\begin{align*} 
	s_{1,a}=s_1-s_{1,b}-s_{1,c} &\geq n- 15\delta n +12 -2\left(\frac{n}{3}+\delta n-3\right) =\frac{n}{3}-17\delta n +18.
	\end{align*}
	Similarly $s_{1,b}$ and $s_{1,c}$ satisfy the same lower bound. Set $A=\{y\in V\setminus S:\ S_y= S^{1,a}\}\cup \{a\}$, $B=\{y\in V\setminus S:\ S_y= S^{1,b}\}\cup\{b\}$ and $C=\{y\in V\setminus S:\ S_y= S^{1,c}\}\cup \{c\}$.  Set $V'=A\cup B\cup C\cup \{x\}$. The calculations above have established the following lemma.

\begin{lemma}[New version of Lemma~\ref{lemma: A,B,C cover almost everything}]\label{lemma: A,B,C cover almost everything new version}
		\[\vert V'\vert \geq n-15\delta n +16, \quad \text{and} \quad 
		\frac{n}{3}-17\delta n +19 \le \vert A\vert, \vert B\vert, \vert C\vert \le \frac{n}{3}+\delta n -2. \qed
		\]
	\end{lemma}
		Let $\mathcal{B}$ be the collection of $3$-edges of $G$ of the form $xAA, xBB, xCC$ (the `bad' triples). Let $\mathcal{M}$ be the 
	collection of non-edges of $G$ of the form $xAB, xAC, xBC$ (the `missing' triples). Viewing $\mathcal{B}$ and $\mathcal{M}$ as $3$-graphs on $V'$, for two distinct vertices $v_1, v_2\in V'$, we let $d_{\mathcal{B}}(v_1,v_2)$ denote their codegree in $\mathcal{B}$ and $d_{\mathcal{M}}(v_1,v_2)$ their codegree in $\mathcal{M}$. 
	\begin{claim}[New version of Claim~\ref{claim: all bad vx degrees small}]\label{claim: all bad vx degrees small new}
		For every $v\in V'\setminus \{x\}$, $d_{\mathcal{B}}(v,x) \leq 18\delta n -20$.	
	\end{claim}
	\begin{proof}
			Suppose without loss of generality that $v\in A$.  If $v=a$, then $d_{\mathcal{B}}(v,x)=0$ because $G$ contains no 3-edges of the form $x a A$. We thus assume that $v\ne a$. The bad triples for the pair $(v,x)$ are triples of the form $a'vx$ for $a'\in A\setminus \{a,v\}$. Suppose $a'vx\in \mathcal{B}$. Then since there is no $K_4$ in $G$ containing $x$, and since, by the definition of $A$, $a'bx$, $vbx$, $a'cx$ and $vcx$ are all in $G$, it must be the case that both of $a'vb$ and $a'vc$ are missing from $G$. 
			Further if $c'\in C\cap \Gamma(v,x)$ then all of $c'vx, bvx, c'bx$ are in $G$, whence $bc'v$ is absent from $G$. Similarly for any $b'\in B$, at most one of $b'cv$, $b'xv$ is in $G$. Finally since $bc v\not\in E(G)$, $b$ and $c$ are contained in exactly one of $\Gamma(b, v)$, $\Gamma(c, v)$, and $\Gamma(x, v)$. To summarize, a vertex $y$ in $V'$ can lie in at most two of $\Gamma(b,v)$, $\Gamma(c,v)$ and $\Gamma(x,v)$ unless $y$ is in $\Gamma_{\mathcal{B}}(x,v)$ (and lies in exactly one of those joint neighbourhoods) or is in $\{b,c,v\}$ (and lies in at most one of those joint neighbourhoods).
				% YZ replaced "Further if $c'\in C\cap \Gamma(a',x)$ then all of $a'c'x, a'bx, a'cx$ are in $G$, whence $a'bc'$ is absent from $G$. Similarly for any $b'\in B$, at most one of $a'b'c$, $a'b'x$ is in $G$."
			Together with our codegree assumption, this gives us
			\begin{align*}
			3\, \left(\frac{2n}{3}-\delta n\right) &\leq d(b,v)+ d(c,v)+d(x,v)\leq 2\vert V'\vert-d_{\mathcal{B}}(v,x)-4  +3(n-\vert V'\vert)\\&= 3n -\vert V'\vert- 4 -d_{\mathcal{B}}(v,x) \leq 2n +15\delta n-20- d_{\mathcal{B}}(v,x),
			\end{align*}
			where we apply $\vert V'\vert\ge n -15\delta n +16$ from Lemma~\ref{lemma: A,B,C cover almost everything new version}
	 in the last inequality. It follows that $d_{\mathcal{B}}(v,x)\leq 18\delta n -20$, as claimed.		
		\end{proof}

\begin{claim}[New version of claim~\ref{claim: all missing vx degrees small}]\label{claim: all missing vx degree small new}
		For every $v\in V'\setminus \{x\}$, $d_{\mathcal{M}}(v,x) \leq 36\delta n-40$.
	\end{claim}
	\begin{proof}
			Suppose without loss of generality that $v\in A$. Then by the codegree assumption, Claim~\ref{claim: all bad vx degrees small new} and the bound on $\vert A\vert$ from Lemma~\ref{lemma: A,B,C cover almost everything new version} we have
			\begin{align*}
			\frac{2n}{3}-\delta n  \leq d(v,x) & \leq n-1-\vert A\vert +d_{\mathcal{B}}(v,x)-d_{\mathcal{M}}(v,x) \\
			& \leq n-1 - \frac{n}{3}+17\delta n -19  +18\delta n -20 - d_{\mathcal{M}}(v,x),
			\end{align*}
			which gives that $d_{\mathcal{M}}(v,x)\leq 36\delta n -40$ as claimed.
		\end{proof}

\begin{claim}[New version of Claim~\ref{claim: no bad vx triples}]\label{claim: no bad vx triples new}
		Provided $\delta \leq \frac{1}{429}$, for every $y \in V(G)\setminus \{x\}$, $\Gamma(y,x)$ has a non-empty intersection with at most two of the parts $A$, $B$ and $C$.
	\end{claim}
	\begin{proof}
			Let $y \in V(G)\setminus \{x\}$. Set $A_y=A\cap\Gamma(x,y)$, $B_y=B\cap \Gamma(x,y)$ and $C_y=C\cap 	\Gamma(x,y)$.  Suppose none of $A_y$, $B_y$, $C_y$ is empty. Fix $a'\in A_y$. For $b'\in B_y$, if $b'\in 		\Gamma(a', x)$, then $a' b' y\not\in E(G)$ -- otherwise $\{a', b', x, y\}$ spans a copy of $K_4$. Similarly, for $c'\in 	C_y\cap \Gamma(a', x)$, we have $a' c' y\not\in E(G)$. Hence, 
			\begin{align*}
			\frac{2n}{3}-\delta n \leq d(a',y) &\leq  n-2 - \vert B_y \cap \Gamma(a', x)\vert - \vert C_y \cap \Gamma(a', x)\vert.
			\end{align*}		
		Claim~\ref{claim: all missing vx degree small new} gives that $d_{\mathcal{M}}(a',x)\leq 36\delta n-40$.  Consequently,
			\[
			| B_y \cap \Gamma(a', x)| + | C_y \cap \Gamma(a', x) | = |B_y| + |C_y| - d_{\mathcal{M}}(a',x) \ge |B_y| + |C_y| - 36\delta n+40
			\]
			This implies that	
			\[
			\frac{2n}{3}-\delta n \leq  n-2 -\vert B_y\vert - \vert C_y\vert + 36\delta n-40, 
			\]
			which  yields $\vert B_y\vert +\vert C_y\vert \leq \frac{n}{3} +37\delta n-42$. Similarly by considering any vertex $b'\in B_y$ and any vertex $c'\in C_y$ we obtain that 
			\begin{align*}
			\vert A_y \vert +\vert C_y\vert \leq \frac{n}{3}+37\delta n -42 \quad  \text{and} \quad \vert A_y \vert +\vert B_y\vert \leq \frac{n}{3} + 37\delta n -42.
			\end{align*}
			Summing these three inequalities and dividing by $2$, we obtain that
			\[
			\vert A_y \vert + \vert B_y\vert + \vert C_y \vert \leq \frac{n+111\delta n-126}{2}.
			\]
			Furthermore, by the codegree condition, 
			\[
			\frac{2n}{3}-\delta n \leq d(x,y) \leq \vert A_y\vert+\vert B_y\vert + \vert C_y\vert  + \left(n-\vert V'\vert\right)
			\leq \frac{n+141\delta n-158}{2},
			\]
			where we apply $\vert V'\vert\ge n -15\delta n +16$ from Lemma~\ref{lemma: A,B,C cover almost everything new version}. Rearranging terms yields $\frac{(1-429\delta)n}{6} \leq -\frac{158}{2}$, which is a contradiction as $\delta \leq 1/429$.
		\end{proof}
			Set $V_1= \{y \in V\setminus \{x\}: \ \Gamma(x,y)\cap A=\emptyset \}$, $V_2=\{y \in V\setminus \{x\}: \ \Gamma(x,y)\cap B=\emptyset \}$ and $V_3 =\{y \in V\setminus \{x\}: \ \Gamma(x,y)\cap C=\emptyset \}$. 
			Without loss of generality, assume that $|V_1| \le |V_2|\le |V_3|$.
		Claim~\ref{claim: no bad vx triples new} shows that $V_1\cup V_2\cup V_3$ covers $V(G)\setminus \{x\}$. We now show that in fact $V_1, V_2, V_3$ are pairwise disjoint, and $A\subseteq V_1$, $B\subseteq V_2$, and $C\subseteq V_3$.  
		Suppose instead that there exists $y\in V_1\cap V_2$. Then $\Gamma(x,y)\cap (A\cup B) =\emptyset$. By the codegree condition and Lemma~\ref{lemma: A,B,C cover almost everything new version}, 
		\[
		\frac{2n}{3}-\delta n \leq d(x,y) \leq \vert C_y\vert  + \left(n-\vert V'\vert\right) \le \frac{n}{3}+\delta n -2 +15\delta n-16=\frac{n}{3} +16\delta n -18.
		\]	
		Rearranging terms yields $\frac{(1 -51\delta)n}{3}\leq -18$, which for $\delta \leq 1/51$ is a contradiction.

		Furthermore, consider $a'\in A$. By Claim~\ref{claim: all missing vx degree small new}, $a'xv\in E(G)$ for all but at most $36\delta n-40$ vertices $v\in B\cup C$. By Lemma~\ref{lemma: A,B,C cover almost everything new version},
			\begin{align*}
			\vert B \vert -36\delta +40 \geq \frac{(1-159\delta)n}{3} +59
			\end{align*}
			which is strictly positive when $\delta \leq 1/159$. Thus we have that $\Gamma(a',x)$ has a non-empty intersection with $B$; similarly we have that $\Gamma(a',x)\cap C\neq \emptyset$, from which we can finally deduce by Claim~\ref{claim: no bad vx triples new} that $\Gamma(a',x)\cap A=\emptyset$ and that $A\subseteq V_1$. Similarly we have $B\subseteq V_2$ and $C\subseteq V_3$. 
			
	We claim that 			
			\begin{equation} \label{eq:Viup}
			\forall i\in \{1,2,3\}, \quad \vert V_i\vert \leq n/3+\delta n -1
			\end{equation}
	Indeed, let $c' \in C$. By the definition of $V_3$, we have $\Gamma(c', x)\subseteq V_1\cup V_2$. By the codegree assumption, it follows that
			\begin{align*} 
			\frac{2n}{3}-\delta n \leq d(c', x)& \leq \vert V_1\vert +\vert V_2\vert = n-1 -\vert V_3\vert, 
			\end{align*}
	from which we get that $\vert V_3\vert \leq n/3+\delta n -1 $, as claimed.  By \eqref{eq:Viup}, we have $\vert V_i\vert=n-1- \bigcup_{j\ne i} \vert V_{j}\vert \geq n/3 -2\delta n +1$ and consequently,
	\[
	\frac43 - 2\delta n \le |V_i| - \frac{n-1}3 \le  \delta n - \frac23.
	\]
	This gives Part (v) of Theorem~\ref{theorem: K4 stability}.
					
		\begin{claim}[New version of Claim~\ref{claim: all missing x degrees small}]\label{newclaim: missing x degrees small, bad x degree zero}
		Let $y\in V_i$. Then provided $\delta\leq \frac{1}{429}$, $\Gamma(y, x)$ contains all but at most $18\delta n -18$  vertices from $\bigcup_{j\ne i} V_j$ and no vertex from $V_{i}$.
	\end{claim}
		\begin{proof}
			Suppose without loss of generality that $y\in V_1$. Then by Claim~\ref{claim: no bad vx triples new}, $A\cap \Gamma(y,x)=\emptyset$. Thus
			\begin{align*}
			\frac{2n}{3}-\delta n \leq d(x,y)\leq \vert \Gamma(x,y)\cap \left(V_2\cup V_3\right)\vert + \vert \Gamma(x,y) \cap \left(V_1\setminus A\right)\vert
			\leq \vert \Gamma(x,y)\cap \left(V_2\cup V_3\right)\vert +15\delta n -16
			\end{align*}
			since $\vert V_1\setminus A \vert \le n - \vert V'\vert\le 15\delta n -16$ by Lemma~\ref{lemma: A,B,C cover almost everything new version}. Hence $\vert \Gamma(x,y)\cap \left(V_2\cup V_3\right)\vert \ge \frac{2n}{3} - 16\delta n +16$. By \eqref{eq:Viup}, %Since $|V_i|\le n/3 +\delta n -1$ for all $i$,
			\[
			| \left(V_2\cup V_3\right) \setminus \Gamma(x,y) | \le 2 \left(\frac{n}{3}+ \delta n - 1 \right) - \left(\frac{2n}{3} - 16\delta n +16 \right) = 18\delta n -18.
			\]
			This establishes the first part of our claim. 
					
			For the second part of our claim (namely that $\Gamma(y,x)\cap V_1=\emptyset$), suppose that  $yy'x\in E(G)$ for some $y'\in V_1$. Then $\Gamma(y, y') \cap \Gamma(y,x)\cap \Gamma(y',x) = \emptyset$. Consequently,
			\begin{align*}
			\frac{2n}{3}-\delta n \leq d(y, y') &\leq 1 + \vert V_1\vert-2 + \left\vert\left( V_2\cup V_3\right)\setminus\left(\Gamma(y,x)\cap \Gamma(y',x)\right)\right\vert \\
			& \leq 1+ \frac{n}{3}+\delta n-1-2+ 2 (18\delta n -18) 
			\end{align*}
			where in the last inequality we applied \eqref{eq:Viup} and the first part of the claim. This implies that $\frac{(1-114\delta)n}{3}\le -38$, a contradiction when $\delta \le 1/114$.
		\end{proof}

	This establishes Part (i) of Theorem~\ref{theorem: K4 stability}. %Parts (ii)--(iv) come from the codegree assumption, Claim~\ref{newclaim: missing x degrees small, bad x degree zero} and counting. 
	By Claim~\ref{newclaim: missing x degrees small, bad x degree zero}, the total number of missing $xV_iV_{j}$ edges, $i\ne j$ is at most 
	\begin{align*}
\frac{1}{2}\sum_i (18\delta n - 18) \vert V_i\vert&= \frac12 (n-1)(18\delta n-18)< 9\delta n^2, 
	\end{align*}
	establishing Part (ii) of the theorem. Since a triple $v_1v_2v_3$ with $v_i \in V_i$ is an edge of $G$ only if one of the $xv_iv_{i+1}$ triples is missing, by Part (ii) and Equation \eqref{eq:Viup}, there can be at most $9\delta n^2 (n/3 + \delta n) < 4\delta n^3$ such triples in total, establishing Part (iii) of the Theorem. 
	
	Finally we need to bound the number of non-edges of $G$ intersecting exactly two of $V_1, V_2, V_3$. Fix $v_1\in V_1$ and $v_2\in V_2$. Given a set $S\subseteq V(G)$, let $d(v_1, v_2, S)= |\Gamma(v_1, v_2) \cap S|$ denote the number of neighbours of $v_1$ and $v_2$ in $S$ and $\bar{d}(v_1, v_2, S)= |S\setminus \Gamma(v_1, v_2)|$. By the codegree condition, we have
	\[
	d(v_1, v_2, V_1\cup V_2)  \ge \frac23 n - \delta n - 1 - d(v_1, v_2, V_3).
	\]
	Together with \eqref{eq:Viup}, this implies that
	\begin{align*}
	\bar{d}(v_1, v_2, V_1\cup V_2) &\le |V_1| + |V_2| - 2 -\left(\frac23 n - \delta n - 1- d(v_1, v_2, V_3)\right)  \\
	& \le 2 \left(\frac{n}3 + \delta n - 1\right) - 2 - \frac23 n + \delta n  +1 + d(v_1, v_2, V_3)  \\
	&= 3\delta n - 3 + d(v_1, v_2, V_3), 
	\end{align*}
The number of non-edges of $G$ in the form of $V_1 V_1 V_2$ or $V_1 V_2 V_2$ is thus
	\begin{align*}
	\frac12 \sum_{v_1\in V_1, v_2\in V_2} \bar{d}(v_1, v_2, V_1\cup V_2) & \le \frac12 \Bigl( |V_1| |V_2| ( 3\delta n -3) + e(V_1, V_2, V_3) \Bigr) 
	%\\ & \le \frac12 |V_1| |V_2| (3\delta n -3) + \frac32 \delta n^3 (1 + 3\delta),
	\end{align*}
	where $e(V_1, V_2, V_3)$ denotes the number of tripartite edges. We know that   
	$e(V_1 ,V_2, V_3)\le 9\delta n^2 (n/3 + \delta n) $. Thus the number of non-edges of $G$ intersecting exactly two of $V_1, V_2, V_3$ is at most	
	\begin{align*}
	\frac12  \sum_i \vert V_i\vert \vert V_{i+1}\vert  \left(3\delta n -3\right) + \frac32 e(V_1, V_2, V_3)
	 &\leq \frac{n-1}{2}\left(\frac{n}{3}+\delta n-1\right)\left(3\delta n-3\right) +\frac92 \delta n^3 (1 + 3\delta) \\
	% & < \frac32 \delta n^3 \left(\frac13 + \delta + 3 + 9\delta \right) <6\delta n^3,
	& < \frac{n}2 \left(\frac{n}{3}+\delta n \right)3\delta n +\frac92 \delta n^3 (1 + 3\delta) < 6\delta n^3,
	\end{align*} 
	where we applied \eqref{eq:Viup} in the second inequality. This establishes Part (iv) of the Theorem.
	\end{proof}
%In this subsection, we prove Theorem~\ref{theorem: K4 threshold} and characterize the extremal configurations.
\begin{proof}[Proof of Theorem~\ref{theorem: extremal configurations}]

\noindent \textbf{Case 1: $n=3m \geq 858$.} 
	Let $G$ be a $3$-graph on $n=3m$ vertices with $\delta_2(G)= 2m-2=\left(\frac{2}{3}-\frac{2}{n}\right)n$. Suppose $x\in V(G)$ is not covered by any copy of $K_4$. Since $\delta=\frac{2}{n}\leq \frac{1}{429}$, we can apply Theorem~\ref{theorem: K4 stability} to obtain a tripartition $V_1\sqcup V_2 \sqcup V_3 =V(G)\setminus \{x\}$ satisfying conditions (i)--(v) from Theorem~\ref{theorem: K4 stability}. Assume without loss of generality that $\vert V_1\vert \leq \vert V_2\vert \leq \vert V_3\vert$. For any vertex $v_3\in V_3$, we have (by condition (i))
	\begin{align*}
	2m-2\leq d(x,v_3)\leq \vert V_1\vert +\vert V_2\vert =3m -1-\vert V_3\vert,
	\end{align*}
	from which it follows that $\vert V_3 \vert \leq m+1$. %This in turn implies that $\vert V_1\vert \geq m-2$.

	Suppose $\vert V_3\vert=m+1$. The condition (i) tells us that all triples of the form $xvv_3$ with $v\in V_1\sqcup V_2$ and $v_3 \in V_3$ must be in $G$, for otherwise $d(x,v_3)<2m-2=\delta_2(G)$. Now consider any pair of vertices $v_1\in V_1$, $v_2\in V_2$ for which $xv_1v_2$ is in $G$ (such pairs must exist by condition (ii), say). For every $v_3\in V_3$, both of $xv_1v_3$ and $xv_2v_3$ are in $G$, whence the tripartite triple $v_1v_2v_3$ must be absent from $G$ (since otherwise $xv_1v_2v_3$ would induce a copy of $K_4$ in $G$). Thus the codegree of $v_1, v_2$ is at most $\vert V_1\vert +\vert V_2 \vert -1=2m-3$, a contradiction.

	Thus $\vert V_3\vert=m$, whence $\vert V_2\vert=m$ also and $\vert V_1\vert=m-1$. Now, by condition (i), for every $v_3\in V_3$ at most one triple $vv_3x$ with $v\in V_1\sqcup V_2$ can be missing in $G$, as otherwise $d(v_3,x)<2m-2$; similarly for every $v_2\in V_2$ at most one triple $vv_2x$ with $v\in V_1\sqcup V_3$ is missing, and for every $v_1\in V_1$ at most $2$ triples $vv_1x$ with $v\in V_2\sqcup V_3$ are missing. Further a tripartite triple $v_1v_2v_3$ can be included in $G$ only if one of the triples $xv_1v_2$, $xv_2v_3$, $xv_1v_3$ is missing from $G$. This shows that $G$ must be (isomorphic to) a subgraph of $F_1(\mathcal{E}, 3m)$ for some admissible collection of pairs $\mathcal{E}$.

\medskip
\noindent \textbf{Case 2: $n=3m+1 \geq 715$.} Let $G$ be a $3$-graph on $n=3m+1$ vertices with minimum codegree $\delta_2(G)= 2m-1=\left(\frac{2}{3}-\frac{5}{3n}\right)n$. Suppose $x\in V(G)$ is not covered by any copy of $K_4$. Since $\delta=\frac{5}{3n}\leq \frac{1}{429}$, we can apply Theorem~\ref{theorem: K4 stability} to obtain a tripartition of $V(G)\setminus \{x\}=V_1\sqcup V_2 \sqcup V_3$ satisfying conditions (i)--(v) from Theorem~\ref{theorem: K4 stability}. Assume without loss of generality that $\vert V_1\vert \leq \vert V_2\vert \leq \vert V_3\vert$.

	For any vertex $v_3\in V_3$, we have (by condition (i))
	\begin{align*}
	2m-1\leq d(x,v_3)\leq \vert V_1\vert +\vert V_2\vert =3m-\vert V_3\vert,
	\end{align*}
	from which it follows that $\vert V_3 \vert \leq m+1$. If $\vert V_3\vert=m+1$, then by the codegree assumption and condition (i) all triples of the form $xvv_3$ with $v_3\in V_3$ and $v\in V_1\sqcup V_2$ are in $E(G)$. Now consider vertices $v_1\in V_1$ and $v_2\in V_2$ for which $xv_1v_2\in E(G)$ (which must exist by condition (ii), say). The tripartite triple $v_1v_2v_3$ does not lie in $E(G)$ for any $v_3\in V_3$, since otherwise $xv_1v_2v_3$ would induce a copy of $K_4$. Thus  $d(v_1,v_2)\leq \vert V_1 \vert +\vert V_2\vert -1=2m-2<2m-1$, a contradiction. We must thus have $\vert V_1\vert= \vert V_2 \vert = \vert V_3\vert=m$.

	Now by condition (i) and the codegree assumption, for every vertex $v_i\in V_i$ all but at most $1$ of the triples $xvv_i$ with $v\in V(G)\setminus \left(V_i\cup\{x\}\right)$ must be in $E(G)$. Furthermore a tripartite triple $v_1v_2v_3$ can belong to $G$ only if one of the triples $xv_1v_2$, $xv_2v_3$, $xv_3v_1$ is absent from $G$. This shows that $G$ is (isomorphic to) a subgraph of $F_1(\mathcal{E}, 3m+1)$ for some admissible collection of pairs $\mathcal{E}$.

\medskip
\noindent \textbf{Case 3: $n=3m+2 \geq 1001$.}

	Let $G$ be a $3$-graph on $n=3m+2$ vertices with minimum codegree $\delta_2(G)= 2m-1=\left(\frac{2}{3}-\frac{7}{3n}\right)n$. Suppose $x\in V(G)$ is not covered by any copy of $K_4$. Since $\delta=\frac{7}{3n}\leq \frac{1}{429}$, we find a tripartition of $V(G)\setminus \{x\}=V_1\sqcup V_2 \sqcup V_3$ satisfying conditions (i)--(v) from Theorem~\ref{theorem: K4 stability}. Assume without loss of generality that $\vert V_1\vert \leq \vert V_2\vert \leq \vert V_3\vert$.

	For any vertex $v_3\in V_3$, by condition (i), we have
	\begin{align*}
	2m-1\leq d(x,v_3)\leq \vert V_1\vert +\vert V_2\vert =3m+1-\vert V_3\vert,
	\end{align*}
	from which it follows that $\vert V_3 \vert \leq m+2$. If $\vert V_3\vert=m+2$, then by the codegree assumption and condition (i) all triples of the form $xvv_3$ with $v_3\in V_3$ and $v\in V_1\sqcup V_2$ are in $E(G)$. Now consider vertices $v_1\in V_1$ and $v_2\in V_2$ for which $xv_1v_2\in E(G)$ (which must exist by condition (ii), say). The tripartite triple $v_1v_2v_3$ does not lie in $E(G)$ for any $v_3\in V_3$, since otherwise $xv_1v_2v_3$ would induce a copy of $K_4$. Thus  $d(v_1,v_2)\leq \vert V_1 \vert +\vert V_2\vert -1=2m-2<2m-1$, a contradiction. We must thus have $\vert V_3\vert\leq m+1$. Our assumption that $\vert V_1\vert\leq \vert V_2\vert \leq \vert V_3$ then implies that $\vert V_3\vert \leq m+1$ and that $\vert V_2\vert \in \{m,m+1\}$. We have two subcases to consider.

	\noindent \textbf{Case 3a: $\vert V_1\vert =\vert V_2\vert =m$.} Condition (i) and the codegree assumption together imply that for every $v\in V_i$, all but at most $2$ of the triples of the form $xv\left(V\setminus\left(V_i\cup\{x\}\right)\right)$ must be in $E(G)$ if $i\in\{1,2\}$, and all but at most $1$ if $i=3$. Further a tripartite triple $v_1v_2v_3$ can be in $E(G)$ only if one of $xv_1v_2$, $xv_2v_3$, $xv_3v_1$ is absent from $E(G)$. This shows that $G$ must be (isomorphic to) a subgraph of $F_1(\mathcal{E}, 3m)$ for some admissible collection of pairs $\mathcal{E}$.

	\noindent \textbf{Case 3b: $\vert V_1\vert =m-1$, $\vert V_2\vert =m+1$.} Condition (i) and the codegree assumption together imply that for every $v\in V_i$, all but at most $1$ of the triples of the form $xv\left(V\setminus\left(V_i\cup\{x\}\right)\right)$ must be in $E(G)$ if $i\in\{2,3\}$, and all but at most $3$ if $i=1$. Further a tripartite triple $v_1v_2v_3$ is in $E(G)$ only if one of $xv_1v_2$, $xv_2v_3$, $xv_3v_1$ is absent from $E(G)$. This shows that $G$ must be (isomorphic to) a subgraph of $F'_1(\mathcal{E}, 3m)$ for some admissible collection of pairs $\mathcal{E}$.
\end{proof}

\end{document}